\documentclass[12pt]{article}
\usepackage[utf8]{inputenc}
\usepackage[T1]{fontenc}
\usepackage{authblk}
\usepackage{amsthm}
\usepackage{amsmath}
\usepackage{amsfonts}
\usepackage{amssymb,mathdots}
\usepackage{amssymb}
\usepackage{graphicx,float,subfigure,multirow}
\usepackage[normalem]{ulem}

\theoremstyle{plain}
\newtheorem{theorem}{Theorem}[section]
\newtheorem{corollary}[theorem]{Corollary}
\newtheorem{proposition}[theorem]{Proposition}
\newtheorem{lemma}[theorem]{Lemma}
\newtheorem{conjecture}[theorem]{Conjecture}

\theoremstyle{definition}
\newtheorem{definition}[theorem]{Definition}
\newtheorem{example}[theorem]{Example}
\newtheorem{notation}[theorem]{Notation}

\newtheorem{remark}[theorem]{Remark}

\newtheoremstyle{case}{}{}{}{}{}{:}{ }{}
\theoremstyle{case}

\usepackage{amsthm}

\usepackage{tikz}
\usetikzlibrary{calc, arrows, positioning}
\tikzstyle{loosely dotted}= [dash pattern=on \pgflinewidth off 4pt]

\oddsidemargin 
\evensidemargin
\usepackage{hyperref}

\newenvironment{packed_enum}{
\begin{enumerate}
  \setlength{\itemsep}{1pt}
  \setlength{\parskip}{0pt}
  \setlength{\parsep}{0pt}
}{\end{enumerate}}
 
\begin{document}

\title{Identifiability of Linear Compartmental Models: \\
The Effect of Moving Inputs, Outputs, and Leaks}
\author[1]{Seth Gerberding\thanks{Email: {\tt seth.gerberding@coyotes.usd.edu}}}
\affil[1]{\small{University of South Dakota}}

\author[2]{Nida Obatake}
\affil[2]{\small{Texas A\&M University}}

\author[2]{Anne Shiu}

%\date{October 29, 2019}
 \date{\today}
\maketitle

\begin{abstract}
A mathematical model is identifiable if its parameters can be recovered from data. 
Here we investigate, for linear compartmental models, whether 
% {\color{blue}
(local, generic)
% } 
identifiability is preserved when parts of the model -- specifically, inputs, outputs, leaks, and edges -- are moved, added, or deleted. 
Our results are as follows.
First, for certain catenary, cycle, and mammillary models, moving or deleting the leak preserves identifiability.  Next, for cycle models with up to one leak, moving inputs or outputs preserves identifiability.  Thus, every cycle model with up to one leak (and at least one input and at least one output) is identifiable.  
% {\color{blue}
Next, we give conditions
under which adding leaks 
%, related to number and location, for which adding more than one leak 
renders a cycle model unidentifiable.
% }
Finally, for certain cycle models with no leaks, adding specific edges again preserves identifiability.  
Our proofs, which are algebraic and combinatorial in nature, rely on results on elementary symmetric polynomials and the theory of input-output equations for linear compartmental models. 
\end{abstract}

\section{Introduction} \label{sec:intro}
Linear compartmental models are a staple in many biological fields, including pharmacology, ecology, and cell biology. These models describe how something, whether it be drug concentration or toxins, moves within a system. In this work, we focus on the 
% {\color{blue}
(local, generic)
% } 
\textit{identifiability} of linear compartmental models, that is, the ability to recover flow parameters from data. 
%This property is valuable for many reasons. One theoretical goal is to find simpler methods to identify parameters in a model. Before answering \textit{how} to identify the parameters, we must answer whether we \textit{can}; this is is the focus of this work. 

Our motivation is the work of Gross, Harrington, Meshkat, and Shiu~\cite{GHMS19One}, 
which investigated the effect on identifiability of adding or deleting some component -- an input, output, leak, or edge -- of a linear compartmental model.  
%We build upon prior work in [1,2,3], which investigated some of the model operations which may affect identifiability. That is, given a model that is identifiable, does changing the model in some way preserve identifiability, or is the new model unidentifiable? 
Two key problems that they raised are as follows.  
First, it is not %always the case that adding edges to a model preserves identifiability, and it is not 
well understood which edges can be added to a model so that identifiability is preserved~\cite[Table~2]{GHMS19One}.
% {\color{blue}\cite[Table~2]{GHMS19One}}. %Example~5.3
%Therefore, it has to be proved that adding particular edges preserves identifiability. 
Second, the effect of removing a leak remains an open question: the authors conjectured that removing a leak from certain models preserves identifiability~\cite[Conjecture~4.5]{GHMS19One}.
% {\color{blue}\cite[Conjecture~4.5]{GHMS19One}}. 

Here we address both problems: we resolve the leak-removal conjecture for three infinite families of models (Theorem~\ref{Catenary, cycle, mam., w/wo leak, ID}), and prove that adding certain edges to a cycle model preserves identifiability 
(Theorems~\ref{thm:fin-wing-0-1-leak-id} and~\ref{thm:fin-wing-partial}).  
% {\color{blue}
We also analyze how the position and number of leaks affects identifiability (Theorem~\ref{thm:cycle-many-leaks}). 
% }
Finally, by analyzing a new operation on models -- namely, moving the output -- we prove that 
every cycle model with up to one leak, at least one input, and at least one output (in any compartment) is identifiable (Theorem~\ref{Big Cycle Model theorem }). 
Our main results are summarized in Table~\ref{tab:results-summary}.

\begin{table}[h!]
    \centering
    \begin{tabular}{p{0.43\textwidth}|p{0.25\textwidth}|c}
    \hline
        {\bf Linear compartmental model} % $M$ 
            & {\bf Operation} % to obtain $\widetilde{M}$ 
            & {\bf Result} \\ \hline\hline
        % No leaks {\color{blue} (delete?)} & add 1 leak & \cite{GHMS19One} (Proposition~\ref{prop:add-leak-in-out}) \\ \hline\hline
        
        % any $M$ & add one input or one output & $M$ and $\widetilde{M}$ are identifiable & Proposition~\ref{prop:add-leak-in-out}\\ \hline
        
        % \multirow{2}{*}{Any  {\color{blue} (delete?)}} & add 1 input & \multirow{2}{*}{\cite{GHMS19One} (Proposition~\ref{prop:add-leak-in-out})}\\ \cline{2-2}
        % & add 1 output & \\ \hline\hline

        % catenary, cycle, mammillary model with $In=Out=\{1\}$, $Leak=\emptyset$ & (add leak in any compartment) & Theorem~\ref{Catenary, cycle, mam., w/wo leak, ID}\\ \hline
%        catenary with $In=Out=\{1\}$, $Leak=\emptyset$ & \multirow{3}{*}{add one leak} % anywhere
%        & \multirow{3}{*}{Theorem~\ref{Catenary, cycle, mam., w/wo leak, ID}}\\ 
%        cycle with $In=Out=\{1\}$, $Leak=\emptyset$ &  & \\
%        mammillary with $In=Out=\{1\}$, $Leak=\emptyset$ &  & \\ \hline\hline
        Catenary, cycle, or mammillary  & \multirow{2}{*}{add 1 leak} % anywhere
        & \multirow{2}{*}{Theorem~\ref{Catenary, cycle, mam., w/wo leak, ID}}\\ 
        with $In=Out=\{1\}$, $Leak=\emptyset$ &  & \\ \hline \hline
        % cycle model with $|In|=1, |Out|=1, Leak=\emptyset$& cycle model with $|In|\geq 1$, $|Out|\geq 1$, $|Leak|\leq 1$ &  Theorem~\ref{Big Cycle Model theorem }\\
        \multirow{6}{*}{ Cycle with $|In|=|Out|=1, Leak=\emptyset$}
        % & 
        % \sout{add 1 or more inputs}{\color{blue}??} &  \multirow{3}{*}{Theorem~\ref{Big Cycle Model theorem }}\\ \cline{2-2}
        % & \sout{add 1 or more outputs}{\color{blue}??} &  \\ \cline{2-2}
        %\cline{2-3}
        & add 1 leak,  move input, move output & \multirow{2}{*}{Theorem~\ref{Big Cycle Model theorem }} \\ 
        \cline{2-3}
        & add all incoming edges & \multirow{2}{*}{Theorem~\ref{thm:fin-wing-0-1-leak-id}}\\ \cline{2-2}
        & add all outgoing edges & \\ \cline{2-3}
        & add 1 incoming edge & \multirow{2}{*}{Theorem~\ref{thm:fin-wing-partial}}\\ \cline{2-2}
        & add 1 or 2 outgoing edges & \\ \hline
    \end{tabular}
    \caption{Summary of results on operations preserving identifiability.  
    For a strongly connected linear compartmental model $M$ with at least one input, if $\widetilde{M}$ is obtained from $M$ by the specified operation, and $M$ is identifiable, then $\widetilde{M}$ is identifiable.  Related prior results, from~\cite{GHRS}, are summarized in Proposition~\ref{prop:add-leak-in-out}. 
    % {\color{blue} 
    Results on how the operation of adding leaks causes identifiability to be lost appear in Section~\ref{sec:add-leaks-to-cycle}. 
    % }
    }
    \label{tab:results-summary}
\end{table}

% {\color{blue} 
Our work fits into the larger 
goal of determining
how properties of ODE systems defined by graphs are affected
by operations on the graphs.  
See, for instance, recent work on a property closely related to identifiability, namely, observability~\cite{chapman-et-al}.
% }

Our proofs harness the theory of input-output equations for linear compartmental models~\cite{bearup, daisy,glad,GHMS19One,MSE15Two,Ovchinnikov-Pogudin-Thompson}.  As a result, our analyses are largely linear-algebraic and combinatorial. For instance, we apply results on elementary symmetric polynomials. 

%\par Our first result (Lemma \ref{lemma1}) analyzes elementary symmetric polynomials: they become essential in proving the later results. The second result (Theorem~\ref{Catenary, cycle, mam., w/wo leak, ID}) addresses leak deletion in cycle, catenary, and mammillary models: we prove that deleting a leak preserves identifiability. Next, in Theorem \ref{Big Cycle Model theorem }, we prove that a cycle model with one input, one output, and one leak in \textit{any compartment} is identifiable. We do so by analyzing a new operation: moving the output. Our final result (Theorem \ref{thm:hybrid-models-id}) proves that adding certain edges, which we term ``interior edges'' in a cycle model also preserves identifiability. 

This paper is structured as follows. Section~\ref{sec:background} introduces the definitions and tools we use throughout the rest of the paper. 
In Section~\ref{sec:results}, we state and prove our main results.
%: results on elementary symmetric polynomials, some cases where leak deletion preserves identifiability, a cycle model with one input one output, and one leak is generically locally identifiable, and that adding certain interior edges preserves identifiability. 
We end with a discussion in Section~\ref{sec:discussion}.
%discusses avenues for further research as well as some notes which may help with future inquiry. 

\section{Background} \label{sec:background}
We begin with some definitions and important preliminary results, following the notation of~\cite{MSE15Two}. Specifically, we focus on linear models and identifiability. 

\subsection{Linear Compartmental Models}

\par A \textit{linear compartmental model} %(\emph{model}, for short) 
consists of a directed graph $G=(V,E)$, and three sets, \textit{In, Out, Leak}  $\subseteq V$, which are the Input, Output, and Leak compartments, respectively (a {\em compartment} is a vertex $i\in V$).
Each edge $j\to i$ in $E$ represents the flow or transfer of material from the $j^{th}$ compartment to the $i^{th}$ compartment, with associated parameter (rate constant) $k_{ij}$.
%Every edge , indicated by $k_{ij}$, where $i$ indicates the compartment where the flow is going \textit{to}, and $j$ indicates where the flowing is coming \textit{from}. 
Each leak compartment $j \in Leak$ also has an associated parameter, $k_{0j}$, the rate constant for the outflow or degradation from that compartment.

Next, each input compartment has an external input, $u_i(t)$, which fuels the system. 
%That is, the input compartment is the source of the material. 
The output compartments, on the other hand, are measurable: we are able to know the concentration in these compartments. We always assume $Out \neq \emptyset$, as models without outputs are not identifiable.

Figure~\ref{fig:cycle-model} depicts an $n$-compartment cycle model, with $In=Out=Leak=\{1\}$. 
%(In general, the input, output, and leaks can be in any compartment.)
The input compartment is labeled by ``in,'' the output is indicated by an edge with an empty circle at the end, also labeled with ``out,'' and the leak has an outgoing edge labeled by the leak parameter $k_{01}$.
To drive some intuition, Figure~\ref{fig:cycle-model} could model the flow of a drug in the body. Compartment 1 is the injection site, like an arm or thigh. 
The input is the shot. The output represents a device that measures how much drug is still in the injection site. The other compartments could represent organs; so, the drug goes from the injection site, to, e.g., the heart, the lungs, etc., and then back to the injection site.

\begin{figure}[ht]
%----------------------
% FIGURE: GRAPH
%----------------------
\subfigure[The cycle graph.]{
    \label{fig:cycle-graph}
    \centering
        \begin{tikzpicture}[scale=2]
 %-----------------------------
 % CYCLE
 %-----------------------------
  	\draw (-1,0) circle (0.2);	
 % 	\draw (1,0) circle (0.2);	
  	\draw (-0.5,0.866) circle (0.2);	
  	\draw (0.5,-0.866) circle (0.2);	
  	\draw (-0.5,-0.866) circle (0.2);	
 %	\draw (0.5,0.866) circle (0.2);	
 % ARROWS%
    \draw[->] (-0.85, -0.25) -- (-0.62, -0.64 );
  	 \draw[->]  (-0.2,-0.866) -- (0.2,-0.866) ;
 	 \draw[->]  (0.62, -0.64) -- (0.9,-0.17); 
 	 \draw[loosely dotted,thick] (0.9,0.17) -- (0.5,0.866);
 	 \draw[->]  (0.2,0.866) -- (-0.2,0.866);
 	 \draw[->] (-0.62, 0.64 ) -- (-0.85, 0.25)  ;
% 	 \draw[->] (-1.1, .18) -- (-1.5, .5);
 % RATES
   	 \node[] at (0.0, -1.03) {$k_{32}$};
   	 \node[] at (0.0, 1.03) {$k_{n,n-1}$};
   	 \node[] at (-0.55, -0.42) {$k_{21}$};
   	 \node[] at (-0.55, 0.4) {$k_{1n}$};
   	 \node[] at (0.55, -0.42) {$k_{43}$};
 % LABEL: numbers
     	\node[] at (-1, 0) {1};
     	\node[] at  (-0.5,-0.866) {$2$};
     	\node[] at  (+0.5,-0.866) {$3$};
     	\node[] at (-0.5,0.866) {$n$};
 % LABEL: Model name
     	\node[above] at (-1.1, -1.5) {Cycle};
  % BOX
 \draw (-1.5,-1.5) rectangle (1.4, 1.2);
 \end{tikzpicture}
}
\hfill%
%----------------------
% FIGURE: MODEL
%----------------------
\subfigure[A cycle model.]{
     \label{fig:cycle-model}
    \centering
    	\begin{tikzpicture}[scale=1.8]
 %-----------------------------
 % CYCLE
 %-----------------------------
  	\draw (-1,0) circle (0.2);	
 % 	\draw (1,0) circle (0.2);	
  	\draw (-0.5,0.866) circle (0.2);	
  	\draw (0.5,-0.866) circle (0.2);	
  	\draw (-0.5,-0.866) circle (0.2);	
 %	\draw (0.5,0.866) circle (0.2);	
 % ARROWS%
    \draw[->] (-0.85, -0.25) -- (-0.62, -0.64 );
  	 \draw[->]  (-0.2,-0.866) -- (0.2,-0.866) ;
 	 \draw[->]  (0.62, -0.64) -- (0.9,-0.17); 
 	 \draw[loosely dotted,thick] (0.9,0.17) -- (0.5,0.866);
 	 \draw[->]  (0.2,0.866) -- (-0.2,0.866);
 	 \draw[->] (-0.62, 0.64 ) -- (-0.85, 0.25)  ;
 	 \draw[->] (-1.1, .18) -- (-1.5, .5);
 % RATES
   	 \node[] at (0.0, -1.03) {$k_{32}$};
   	 \node[] at (0.0, 1.03) {$k_{n,n-1}$};
   	 \node[] at (-0.55, -0.42) {$k_{21}$};
   	 \node[] at (-0.55, 0.4) {$k_{1n}$};
   	 \node[] at (0.55, -0.42) {$k_{43}$};
   	 \node[] at (-1.1, .4) {$k_{01}$};
 % LABEL: numbers
     	\node[] at (-1, 0) {1};
     	\node[] at  (-0.5,-0.866) {$2$};
     	\node[] at  (+0.5,-0.866) {$3$};
     	\node[] at (-0.5,0.866) {$n$};
 %OUTPUT
  	\draw (-1.33,-.49) circle (0.05);	
 	 \draw[-] (-1, -.2 ) -- (-1.3, -.45);	
 	 \node[] at (-1.65, -.49) {out};
 % Input
 	 \draw[->] (-1.7, 0) -- (-1.3, 0);	
   	 \node[] at (-1.9, 0) {in};
 % LABEL: Model name
 %     	\node[above] at (-1.8, -1.5) {Cycle};
  % BOX
 \draw (-2.2,-1.5) rectangle (1.4, 1.2);
 \end{tikzpicture}
}
    \caption{The cycle graph with $n$ compartments (cf.~\cite[Figure~2]{GNA17Three}), and a linear compartmental model arising from this graph (with $In=Out=Leak=\{1\}$).}    \label{fig:cycle}
\end{figure}
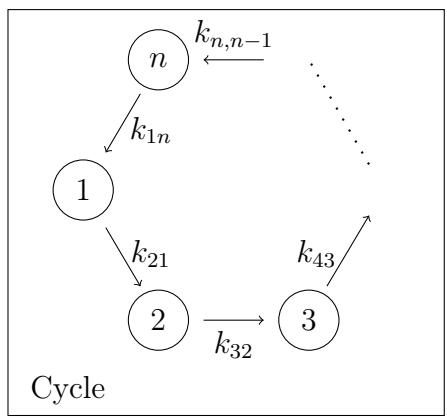
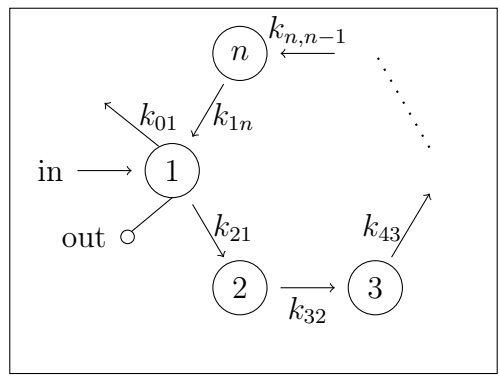

We recall several more definitions:
    \begin{definition}   \label{strongly connected}
    ~
    \begin{enumerate}
        \item 
        A directed graph is \textit{strongly connected} if there exists a directed path from each vertex to every other vertex.
         A directed graph $G$ with $n$ vertices is \textit{inductively strongly connected} with respect to vertex 1 if 
        there exists a reordering of the vertices that preserves vertex 1, such that, 
        for $i=1,2,\dots,n$, the induced subgraph $G_{\{1, 2, \dots, i\}}$ is strongly 
        connected.
        %connected for $i=1,\dots,|V|$ for some ordering of the vertices %$1,\dots,i$ which must start at vertex 1.}
        \item 
        A linear compartmental model $(G, In, Out, Leak)$ is \textit{strongly connected} (respectively, {\em inductively strongly connected} with respect to vertex 1) if \textit{G} is strongly connected (respectively, inductively strongly connected with respect to vertex 1).
    \end{enumerate}
    \end{definition}

For instance, for $n \geq 3$, the cycle model in Figure~\ref{fig:cycle-model} is strongly connected but not inductively strongly connected.
%,as the corresponding graph $G$ is the directed cycle of length $n$.
    
    \begin{definition}
        The \textit{compartmental matrix} of a linear compartmental model $(G, In, Out, Leak)$ with $n$ compartments is the $n \times n$ matrix $A$ given by: 
         $$ A_{ij}:= 
         \begin{cases}
-k_{0i} - \sum_{p: i\to p \in E} k_{pi} \qquad &\text{ if } i=j \text{ and } i \in Leak\\ 
-\sum_{p: i\to p\in E} k_{pi} 
&\text{ if } i=j \text{ and } i \not \in Leak\\
k_{ij} &\text{ if } j \to i \text{ is an edge of } G\\
0 &\text{ otherwise}
\end{cases} $$
        \label{Compartmental matrix definition}
    \end{definition}
    
      Furthermore, a linear compartmental model defines the following system of linear ODE's, with inputs $u_i(t)$ and outputs $y_i(t)$, where $ {x}(t)=(x_1(t), x_2(t), ~\dots ~,x_n(t))$
      is the vector of concentrations of the compartments at time $t$: 
  
\begin{align}
    \label{eq:ode}
    x'(t)~ =&~ Ax(t) + u(t) \\
    \label{eq:ode-2}
    y_i(t)~=&~ x_i(t) \quad \quad \quad \quad  \text{for}~ i \in Out~,
\end{align}

\noindent where $u_i(t)=0$ for $i\not\in In.$

\begin{notation}
Throughout this work, we let $(B)_{ji}$ denote the submatrix obtained from a matrix $B$ by removing row $j$ and column $i$.
\end{notation}

    \subsection{Input-Output Equations}
    For linear compartmental models, \textit{input-output equations} are equations that hold along every solution of ~\eqref{eq:ode}--\eqref{eq:ode-2}, where only the parameters  $k_{ij}$, the input variables $u_i$, the output variables $y_j$, and their derivatives are involved. The following general formulation of these equations comes from Meshkat, Sullivant, and Eisenberg~\cite[Theorem 2]{MSE15Two} (see also~\cite[Remark 2.7]{GHMS19One}):

        \begin{proposition}
    Let $M=(G, In, Out, Leak)$ be a linear compartmental model with n compartments and at least one input. Define $\partial I$ to be the $n \times n$ matrix in which each diagonal entry is the differential operator $d/dt$ and each off-diagonal entry is 0. Let A be the compartmental matrix. Then, the following equations (for $i\in Out$) are input-output equations for $M$:
    \begin{equation} \label{eq:i-o}
        \det(\partial I-A)y_i ~=~ \sum_{j\in In}(-1)^{i+j}\det\left((\partial I-A)_{ji}\right)u_j~,
    \end{equation}
    where $(\partial I-A)_{ji}$ is the % $(n-1) \times (n-1)$ 
    matrix obtained from $(\partial I-A)$ by removing row $j$ and column $i$. 
    \label{Input-Output Equation Definition}
    \end{proposition}
    
    From the input-output equations~\eqref{eq:i-o}, we derive a \textit{coefficient map}, denoted by $c:\mathbb{R}^{|E|+|Leak|}\to \mathbb{R}^m$, which evaluates each vector of parameters 
    $(k_{ij})_{(j,i) \in E,~{\rm or}~ i = 0~{\rm and}~j \in Leak}$ 
    at the vector of non-monic coefficients of the input-output equations.
    Here, $m$ denotes the number of such coefficients. To give a formula for  this number, directly from the model, remains an open question. 
    %Here, $|E|$ is the number of non-leak parameters, or $k_{ij}s$, where $i\neq 0$, and $|Leak|$ is the number of leaks, or $k_{0j}s$, (so, $|E|+|Leak|$ is the total number of parameters). As for the image, $k$ represents the number of coefficients in the input-output equation. 
%    Note that each coefficient is a polynomial in the parameters, e.g., $k_{21}k_{43}+k_{01}k_{13}$. 
    %\textit{A priori, there is no way to know for sure how many coefficients there will be}.

     \begin{example} 
    \label{ex:main}
    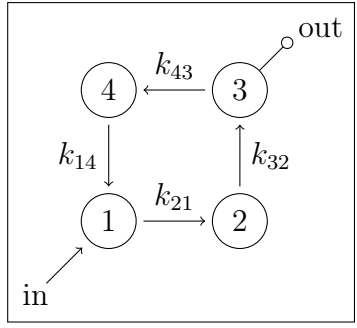
\begin{figure}[ht]
    \centering
    \begin{tikzpicture}[scale=1.8,
    compartment/.style={circle, draw=black, minimum size=7mm}]
 %-----------------------------
 % CYCLE-4
 %-----------------------------
    %COMPARTMENTS%
    \node[compartment]  (one) {1};
    \node[compartment]  (two) [right=of one] {2};
    \node[compartment]  (three) [above=of two] {3};
    \node[compartment]  (four) [above=of one] {4};
    \node[] (output) [node distance=6mm, above right=of three] {};
    \node[] (input) [node distance=8mm, below left=of one] {};

    %ARROWS & RATES%
    \draw[->] ([xshift=.5mm]one.east) -- ([xshift=-.5mm]two.west) node[midway,above] {$k_{21}$};
    \draw[->] ([yshift=.5mm]two.north) -- ([yshift=-.5mm]three.south) node[midway,right] {$k_{32}$};
    \draw[->] ([xshift=-.5mm]three.west) -- ([xshift=.5mm]four.east) node[midway,above] {$k_{43}$};
    \draw[->] ([yshift=-.5mm]four.south) -- ([yshift=.5mm]one.north) node[midway,left] {$k_{14}$};
    
    %OUTPUT
    \draw[-o] (three.north east) -- (output) node[xshift=-2mm,right] {out};
    %INPUT
   	\draw[<-] ([xshift=-.5mm,yshift=-.5mm]one.south west) -- (input) node[] {in};

    % BOX
    \draw ([xshift=-6mm,yshift=-6mm]one.south west) rectangle ([xshift=7mm,yshift=5mm]three.north east);
 \end{tikzpicture}
    \caption{A 4-compartment cycle model with $In=\{1\}$, $Out=\{3\}$, and $Leak=\emptyset$.}
    \label{fig:3comp-ex}
\end{figure}

    For the cycle model shown in Figure~\ref{fig:3comp-ex}, 
    %By Definition~\ref{Compartmental matrix definition}, 
    the compartmental matrix is: 
    \[
        A=\begin{bmatrix}
        -k_{21} & 0 & 0 & k_{14}\\
        k_{21} & -k_{32}& 0 & 0\\
        0 & k_{32} & -k_{43}& 0\\
        0 & 0 & k_{43} & -k_{14}\\
        \end{bmatrix}.
    \]
%    By Proposition \ref{Input-Output Equation Definition}, 
    We have $|Out|=1$, so 
    there is a single input-output equation~\eqref{eq:i-o}:
    %, as $|Out|=1$, only one term on the right hand side since $|In|=1$, and the input-output equation is: 
    
    $$
        \det \left( 
        \begin{bmatrix}
        \frac{d}{dt}+k_{21} & 0 & 0 & -k_{14}\\
        -k_{21} & \frac{d}{dt}+k_{32}& 0 & 0\\
        0 & -k_{32} & \frac{d}{dt}+k_{43}& 0\\
        0 & 0 & -k_{43} & \frac{d}{dt}+k_{14}\\
        \end{bmatrix}
        \right)
        y_3 
        ~=~ 
        \det
        \left(  
        \begin{bmatrix}
        -k_{21} & \frac{d}{dt}+k_{32} & 0\\
        0 & -k_{32} & 0\\
        0 & 0 & \frac{d}{dt}+k_{14}\\
        \end{bmatrix}
        \right)u_1~, 
    $$
    which simplifies as follows: 
    \begin{multline*}
        y_3^{(4)} + 
        y_3^{(3)} (k_{14} + k_{21} + k_{32} + k_{43}) +
        y_3^{(2)} (k_{14} k_{21} + k_{14} k_{32} + k_{21} k_{32} + k_{14} k_{43} + k_{21} k_{43} + k_{32} k_{43})\\
        +y_3^{(1)} (k_{14} k_{21} k_{32} + k_{14} k_{21} k_{43} + k_{14} k_{32} k_{43} + k_{21} k_{32} k_{43})
        ~=~u_1^{(1)} (k_{21} k_{32}) + 
        u_1 (k_{14} k_{21} k_{32})~.
    \end{multline*}
    % \begin{multline*}
    %     {\small  y^{(3)}_1+(k_{01}+k_{13}+k_{21}+k_{32})y^{(2)}_{1}+(k_{01}k_{13}+k_{01}k_{32}+k_{13}k_{21}+k_{13}k_{32}+k_{21}k_{32})y^{(1)}_1+k_{01}k_{13}k_{32}y_1}\\
    %     \small   =u^{(2)}_1+(k_{13}+k_{32})u^{(1)}_1+k_{13}k_{32}u_1.
    % \end{multline*}
    
   \noindent Thus, the coefficient map $c: \mathbb{R}^4 \to \mathbb{R}^5$ is given by:   
\begin{align*}
(k_{21}, k_{32}, k_{43}, k_{14}) \mapsto 
 (k_{14} + k_{21} + k_{32} + k_{43}, 
   ~k_{14} k_{21} + k_{14} k_{32} + k_{21} k_{32} + k_{14} k_{43} + k_{21} k_{43} + k_{32} k_{43}, \\
   ~k_{14} k_{21} k_{32} + k_{14} k_{21} k_{43} + k_{14} k_{32} k_{43} + k_{21} k_{32} k_{43}, 
   ~k_{21} k_{32},
   ~k_{14} k_{21} k_{32}). 
\end{align*}

\end{example}

    \subsection{Identifiability}
    %Next we turn to a key property of a model. 
    A model is (structurally) \textit{identifiable} if all parameters $k_{ij}$ can be recovered from  data~\cite{chis2011,miao}. More precisely, for an identifiable model, we can derive the values of the parameters from perfect (noise-free) input-data and output-data (arising from generic initial conditions). Returning to the injection analogy from the last subsection, the question is: If we know the amount of injected drug and the amount of drug still present in the injection site at any time, \textit{can we determine the rate at which the drug transfers from one organ to the next?} 
    
    There are several kinds of identifiability, but we focus on \textit{generic local identifiability}. 
    To determine whether a model is generically locally identifiable, we analyze coefficient maps.  That is, the following definition of identifiability (for strongly connected models with at least one input) is equivalent to the definition above; this equivalence was proved 
    %Theorem 6.6
    by Ovchinnikov, Pogudin, and Thompson~\cite[Corollary 3.2]{Ovchinnikov-Pogudin-Thompson}\footnote{Ovchinnikov {\em et al.} considered identifiability over $\mathbb{C}$, whereas we consider identifiability over $\mathbb{R}$, but these are equivalent in our setting (cf.\ the discussion in~\cite[\S 3]{baaijens-draisma}).}. 
    
   \begin{definition}
    Let $M=(G, In, Out, Leak)$ be a linear compartmental model 
    that is strongly connected and has at least one input.  Let $c:~ \mathbb{R}^{|E|+|Leak|}\to \mathbb{R}^m$ denote the coefficient map arising from the input-output equations~\eqref{eq:i-o}. Then $M$ is 
    \textit{generically locally identifiable} if, 
    outside a set of measure zero, every point in $\mathbb{R}^{|E|+|Leak|}$ is in some open neighborhood $U$ such that the restriction $c|_{U}: U\to \mathbb{R}^k$ is one-to-one.
%    \begin{enumerate}
%        \item \textit{globally identifiable} if \textit{c} is one-to-one, and is \textit{generically globally identifiable} if \textit{c} is one-to-one outside a set of measure zero. 
%        \item \textit{locally identifiable} if around every point in $\mathbb{R}^{|E|+|Leak|}$ there is an open neighborhood $U$ such that $c: U\to \mathbb{R}^k$ is one-to-one, and is \textit{generically locally identifiable} if, outside a set of measure zero, every point in $\mathbb{R}^{|E|+|Leak|}$ has such an open neighborhood $U$. 
%        \item \textit{unidentifiable} if \textit{c} is infinite-to-one. 
%    \end{enumerate}
    \label{Indentifiability definition}
    \end{definition}
    
Next, we state some prior results that we will use.  The following, due to Meshkat, Sullivant, and Eisenberg~\cite[Proposition 2]{MSE15Two}, 
 is key to many of our proofs.
    \begin{proposition} 
    A linear compartmental model (G, In, Out, Leak), with $G=(V,E)$, is generically locally identifiable if and only if the rank of the Jacobian matrix of its coefficient map c, when evaluated at a generic point, is equal to $|E|+|Leak|$. \label{prop:rank}
    \end{proposition}
    
%    \begin{remark}
%    One important thing to note is that generically local identifiability \textit{does not} guarantee the model will be locally identifiable. There may be values of certain parameters which make the model unidentifiable. For further reading on this issue, see \cite{GNA17Three}.  
%    \end{remark}

\begin{example}[Example~\ref{ex:main}, continued] \label{ex:main-2}
We show that the model from Example~\ref{ex:main} is generically locally identifiable, by
proving that the Jacobian matrix of its coefficient map generically has rank $|E|+|Leak| = 4$. 
%Let $e_1,e_2,e_3,\kappa$ 
Denote the first four non-monic coefficients as follows:
\begin{align*}
    e_1~&=~k_{14} + k_{21} + k_{32} + k_{43}\\
    e_2~&=~k_{14} k_{21} + k_{14} k_{32} + k_{14} k_{43} + k_{21} k_{32} + k_{21} k_{43}+ k_{32} k_{43}\\
    e_3~&=~k_{14} k_{21} k_{32} + k_{14} k_{21} k_{43} + k_{14} k_{32} k_{43} + k_{21} k_{32} k_{43} \\
    \kappa~&=~k_{21}k_{32}~.
\end{align*}
% Take the $(4\times 4)$-submatrix of the Jacobian matrix with rows indexed by $e_1,e_2,e_3,\kappa$ and columns indexed by $k_{21},k_{32},k_{43},k_{14}$:
% {\footnotesize
% \[
%  \left[ \begin{array}{cccc} 1&1&1&1\\ k_{14}+k_{32}+k_{43}&k_{14}+k_{21}+k_{43}&k_{14}
% +k_{21}+k_{32}&k_{21}+k_{32}+k_{43} \\ 
% k_{14}\,k_{32}+k_{14}\,k_{43}
% +k_{32}\,k_{43}&k_{14}\,k_{21}+k_{14}\,k_{43}+k_{21}\,k_{43}&k_{14}\,k_{21}+k_{14}\,
% k_{32}+k_{21}\,k_{32}&k_{21}\,k_{32}+k_{21}\,k_{43}+k_{32}\,k_{43}\\ k_{32}&k_{21}&0&0
% \end{array} \right] .
% \]
% }
For the $(4\times 4)$-submatrix of the Jacobian matrix arising from rows indexed by $e_1,e_2,e_3,\kappa$ and columns $k_{21},k_{32},k_{43},k_{14}$, the determinant is a nonzero polynomial in the $k_{ij}$'s:
\[
-(k_{32}-k_{21}) 
(k_{14}-k_{43}) 
\bigg(
(k_{14}-k_{32})(k_{32}-k_{43})
+k_{14} k_{21}
-k_{21}^2
-k_{21} k_{32}
+k_{21} k_{43}
\bigg)~.
% (
% +k_{14} k_{32}
% -k_{32}^2
% -k_{14} k_{43}
% +k_{32} k_{43}
% +k_{14} k_{21}
% -k_{21}^2
% -k_{21} k_{32}
% +k_{21} k_{43}
% ),
\]
\noindent %which is generically nonzero. 
So, by Proposition~\ref{prop:rank}, the model is generically locally identifiable. 
\end{example}

% \begin{remark}
% In Proposition~\ref{Output-Cycle IDable}, we generalize Example~\ref{ex:main-2}: We will show that \emph{any} $n$-compartment cycle model with no leaks, exactly one input and exactly one output is generically locally identifiable.
% \end{remark}

The following result combines~\cite[Proposition~4.1,  Proposition~4.6, and Theorem 4.3]{GHMS19One}. 
    \begin{proposition}  \label{prop:add-leak-in-out}
    Let $M$ be a linear compartmental model that is strongly connected and has at least one input.  Assume that one of the following holds:
    \begin{enumerate}
        \item $M$ has no leaks, and $\widetilde M$ is a model obtained from $M$ by adding one leak; 
        \item $\widetilde{ {M}}$ is a model obtained from $M$ by adding one input or one output; or
        \item 
        $M$ has an input, an output, and a leak in a single compartment (and no other inputs, outputs, or leaks), and $\widetilde M$ is the model obtained from $M$ by removing the leak.
        \end{enumerate}
    Then, if $M$ is generically locally identifiable, then so is $\widetilde M$. 
        \end{proposition}

\begin{remark} \label{rmk:0-constant-term}
Most models we consider have at most one leak. So, by Proposition~\ref{prop:add-leak-in-out}, 
to prove identifiability, it will suffice to analyze the input-output equations of models without leaks.  For such models, in the left-hand side of the input-output equation~\eqref{eq:i-o}, the constant term is 0 (e.g., see Example~\ref{ex:main}).  Indeed, this term is, up to sign, the determinant of the compartmental matrix $A$, which (if there are no leaks) has 0 column sums. 
\end{remark}

% Need to decide how to work this in. A little odd where I've placed it, because the last remark says that we are only focused on models with at most one leak.
We end this section by 
recalling from~\cite{MeshkatSullivant, MSE15Two}
an important class of models that are identifiable when all leaks (or all but one) are removed (see Proposition~\ref{prop:id-cycle-model-0-1-leaks}).

\begin{definition}\label{def:id-cycle-model}
A linear compartmental model $M=(G,In,Out,Leak)$, with $G=(V,E)$, is an \textit{identifiable cycle model} if
   (1)~$G$ is strongly connected, 
(2)~$In=Out=\{1\}$, 
(3)~$Leak=V$, and
(4)~the dimension of the image of the coefficient map is $|E|+1$.
%\begin{enumerate}
%    \item $G$ is strongly connected,
%    \item $In=Out=\{1\}$
%    \item $Leak=V$,
%    \item The dimension of the image of the coefficient map $c$ is $|E|+1$.
%\end{enumerate}
\end{definition}

When conditions (1)--(3)  
in Definition~\ref{def:id-cycle-model} hold, and also the equality $|E|=2|V|-2$, 
then a sufficient criterion for condition (4) to hold 
is that the graph $G$ is inductively strongly connected with respect to vertex 1~\cite{MeshkatSullivant} (c.f.~\cite[Remark 1]{MSE15Two}).

% {\color{red}
% Tell the reader why we're defining this in the first place (point to Prop~\ref{prop:id-cycle-model-0-1-leaks}).
% }
% {\color{blue}
% \begin{proposition}\cite[Theorem 1]{MSE15Two}\label{prop:id-cycle-model-one-leak}
% Let $M$ be an identifiable cycle model. If the model is changed to have
% exactly one leak, then the resulting model is locally identifiable.
% \end{proposition}
% }

\subsection{Elementary Symmetric Polynomials} \label{sec:elem-sym-poly}
We will use the following lemma to prove identifiability results in the next section.

\begin{lemma} \label{lem:elem-sym-poly}
Let $n$ be a positive integer.
For $1 \leq m \leq n$, let $e_m$ be the $m$-th elementary symmetric polynomial on a set of variables $X=\{x_1, x_2,  \dots ,x_n\}$. 
Let $J(V)$ denote the Jacobian matrix of $V:=\{e_1, e_2, \dots, e_n\}$ with respect to $x_1,x_2, \dots ,x_n$.
Then $\det J(V)$ is a nonzero polynomial in the $x_i$'s.
%, {\color{blue} and therefore every square submatrix of $J(V)$ has nonzero determinant.} {\color{violet} Do we use this last part?  If no, delete.}
\end{lemma}

 \begin{proof}
 %Let $\{\hat{i}\}:=[n]\setminus\{i\}$ {\color{blue} double-check this notation is used} and $\{\hat{x}_i\}:=X \setminus\{x_i\}$. Then 
 For $1 \leq m \leq n$, the $m$-th elementary symmetric polynomial on $X$ is as follows: \\
 \begin{center}
     $e_m ~=~ 
     \sum\limits_{j_{1}<j_{2}<\dots<j_m} x_{j_1}...x_{j_m}
     ~=~
     \sum\limits_{\substack{j_2<...<j_m \\ j_k \neq i}} x_i (x_{j_2}...x_{j_m}) + 
     \sum\limits_{\substack{ l_1<...<l_m \\ l_k \neq i} } x_{l_1}...x_{l_m}$~,
 \end{center}
 for any $1 \leq i \leq n$.
 Thus, taking the partial derivative with respect to $x_i$ yields: 
 \begin{equation}
   \dfrac{\partial e_m}{\partial x_i}~=~ \sum\limits_{
         \substack{j_2<...<j_m \\ j_k \neq i}        
         } x_{j_2}...x_{j_m}~=~e_{m-1}\{\hat{x_i}\}~,
 \end{equation}
 where $e_{m-1}\{\hat{x_i}\}$ is the $(m-1)$-st elementary symmetric polynomial on the set $X 
 \smallsetminus \{x_i\}$. Hence, the Jacobian matrix of $V$ is as follows: 
 \begin{center}
 $J(V)=
 \begin{bmatrix}
 1 & 1 &\dots&1 \\
 e_1\{\hat{x_1}\}& e_1\{\hat{x_2}\} & \dots   & e_1\{\hat{x_n}\} \\
 \vdots & \vdots & \ddots & \vdots \\
 e_{n-1}\{\hat{x_1}\} & e_{n-1}\{\hat{x}_2\}& \dots & e_{n-1}\{\hat{x_n}\}\\
 \end{bmatrix} $.
 \end{center}
 Finally, by equation~(5) in the proof of~\cite[Theorem 5.1]{MSE15Two}, 
 %It will now suffice to show that $\det(J(V)) \neq 0$. We do so similarly as the proof of Theorem 5.1 in \cite{GNA17Three} by showing that 
 $\det(J(V))$ equals, up to sign, the \textit{Vandermonde polynomial} on $X$, which is nonzero:
 %\begin{equation*}
 $        \det(J(V))~=~ \pm \prod_{1\leq i < j \leq n}(x_i-x_j). $ %~. %\label{Vader Poly}
 %\end{equation*}
 \end{proof}

% {\color{blue}
\begin{remark} \label{rem:alternate-proof}
An alternate proof of Lemma~\ref{lem:elem-sym-poly} is as follows.  The Jacobian matrix of $n$ polynomials in $n$ unknowns is nonzero if and only if the polynomials are algebraically independent~\cite[Theorem 2.3]{EhrenborgRota}. 
And the algebraic independence of the elementary symmetric polynomials is known~\cite[pg.\ 20]{Macdonald}.
\end{remark}
% }

\section{Results} \label{sec:results}

In this section we present our main results. 
%In Section 3.1, we discuss elementary symmetric polynomials, as they are tools we use in our proofs. 
In Section~\ref{sec:leak}, we recall a conjecture on removing leaks~\cite{GHMS19One} and then prove the conjecture for three infinite families of models (Theorem~\ref{Catenary, cycle, mam., w/wo leak, ID}). 
Next, in Section~\ref{sec:cycle}, we investigate a new operation: moving the input and output.  
We prove that 
every cycle model with up to~one leak, at least~one input, and at least~one output is identifiable (Theorem \ref{Big Cycle Model theorem }). 
%  {\color{blue} 
 In Section~\ref{sec:add-leaks-to-cycle},
 we give a partial converse to this result by analyzing how the number and location of leaks affects whether a cycle model is identifiable.
%  }
%  {\color{blue} 
 Lastly,
%  } 
 in Section~\ref{sec:new-edges},
 we show that adding certain edges to cycle models preserves identifiability (Theorem~\ref{thm:fin-wing-partial}).

\subsection{Adding or Removing a Leak} \label{sec:leak}
%Our next result pertains to more than just cycle  models: it extends to catenary models (Figure~\ref{fig:cat-model}) and mammillary models (Figure~\ref{fig:mamm-model}) as well.
The following was posed by Gross, Harrington, Meshkat, and Shiu \cite[Conjecture 4.5]{GHMS19One}:

\begin{conjecture}[Removing a leak] \label{conj:remove-leak}
Let $\widetilde{M}$ be a linear compartmental model that is strongly connected and has at least one input and exactly one leak. If $\widetilde{M}$ is generically locally identifiable, then so is the model \textit{M} obtained from $\widetilde{M}$ by removing the leak. 
\end{conjecture}

The next result resolves Conjecture~\ref{conj:remove-leak} for three infinite families of models, which we introduce now.  A {\em catenary model} is a linear compartmental model $(G,In,Out,Leak)$ for which $G$ is the catenary graph in Figure~\ref{fig:cat-model} (for some $n$).  Similarly, a {\em cycle model} (respectively, a {\em mammillary model}) arises from the graph in Figure~\ref{fig:cycle-graph} (respectively, Figure~\ref{fig:mamm-model}).  All three families of models are among the most common in the literature~\cite{distefano-book, godfrey, vajda1982,
van-den-Hof, vicini-mam-cat}.

\begin{figure}[ht]
\subfigure[The catenary graph.]{
    \label{fig:cat-model}
    \centering
        \begin{tikzpicture}[scale=1.8]
 %-----------------------------
 % Catenary
 %-----------------------------
  	\draw (-1,0) circle (0.2);	
    \draw ( 0,0) circle (.2); 
    \draw ( 1, 0) circle (.2);
    \draw (2, 0) circle (.2); 

 % ARROWS%
 \draw[->] (-.75, .1) -- ( -.25, .1); 
 \draw[->] (.25, .1) -- (.75, .1); 
% \draw[loosely dotted, thick, ->] (1.25, .1) -- (1.75, .1);
  \draw[<-] (-.75, -.1) -- ( -.25, -.1); 
 \draw[<-] (.25, -.1) -- (.75, -.1); 
\draw[loosely dotted, thick, -] (1.4, 0) -- (1.6, 0);
 %\draw[loosely dotted, thick, <-] (1.25, -.1) -- (1.75, -.1);
  % RATES
   	 \node[] at (0.5, .25) {$k_{32}$};
   	%  \node[] at (-0.95, -0.5) {$k_{01}$};
   	 \node[] at (-0.5, 0.25) {$k_{21}$};
   	 \node[] at (0.5, -0.25) {$k_{23}$};
   	 \node[] at (-.5, -0.25) {$k_{12}$};
%   	 \node[] at (1.5, .25) {$k_{n, n-1}$};
%   	 \node[] at (1.5, -.25) {$k_{n-1, n}$};
   	 
 % LABEL: numbers
     	\node[] at (-1, 0) {1};
     	\node[] at  (0,0) {$2$};
     	\node[] at  (1,0) {$3$};
     	\node[] at (2,0) {$n$};
     	
%  %OUTPUT
%   	\draw (-1.33,-.49) circle (0.05);	
%  	 \draw[-] (-1, -.2 ) -- (-1.3, -.45);	
%  	 \node[] at (-1.65, -.49) {out};
%  % Input
%  	 \draw[->] (-1.7, 0) -- (-1.3, 0);	
%   	 \node[] at (-1.9, 0) {in};
 % LABEL: Model name
     	\node[above right] at (-1.4, -1.4) {Catenary};
  % BOX
 \draw (-1.4,-1.4) rectangle (2.4, 1.2);
%  \draw (-2.2,-1.5) rectangle (1.4, 1.4);
 \end{tikzpicture}
}
\hfill%
\subfigure[The mammillary graph.]{
     \label{fig:mamm-model}
     \centering
 \begin{tikzpicture}[scale=2]
 %-----------------------------
 % MAMM
 %-----------------------------
  	\draw (-1,0) circle (0.2);	
    \draw (1, 0) circle (.2); 
    \draw (1, 1) circle (.2); 
    \draw (1, -1) circle (.2); 
 % ARROWS%
    \draw[<-] (-.8, .25) -- (.75, .95); 
    \draw[->] (-.75, -.05) -- ( .75, -.05); 
    \draw[->] (-.8, -.25) -- ( .75, -.95);
    
    \draw[->](-.8, .15) -- (.75, .85); 
    \draw[<-] (-.75, .05) -- (.75, .05);
    \draw[<-] (-.8, -.15) -- (.75, -.85);
    
    \draw[loosely dotted, thick] (1, .4) -- ( 1, .6); 
    
    % \draw[->] (-1.2, .22) -- (-1.5, .5);
    
 % RATES
   	 \node[] at (-.25, -.7) {$k_{21}$};
   	 \node[] at (0.0, .8) {$k_{1n}$};
   	 \node[] at (.25, 0.4) {$k_{n1}$};
   	 \node[] at (0, -0.2) {$k_{31}$};
   	 \node[] at (0, .2) {$k_{13}$};
   	 \node[] at (.25, -.45) {$k_{12}$};
   	%  \node[] at (-1.15, 0.5) {$k_{01}$};
 % LABEL: numbers
     	\node[] at (-1, 0) {1};
     	\node[] at (1,-1) {2}; 
        \node[] at (1,1) {$n$}; 
     	\node[] at (1, 0) {3};
%  %OUTPUT
%   	\draw (-1.33,-.49) circle (0.05);	
%  	 \draw[-] (-1, -.2 ) -- (-1.3, -.45);	
%  	 \node[] at (-1.65, -.49) {out};
%  % Input
%  	 \draw[->] (-1.7, 0) -- (-1.3, 0);	
%   	 \node[] at (-1.9, 0) {in};
 % LABEL: Model name
     \node[above right] at (-1.4, -1.4) {Mammillary};
  % BOX
 \draw (-1.4,-1.4) rectangle (1.4, 1.3);
 \end{tikzpicture}
}
\caption{
Two graphs with $n$ compartments (cf.~\cite[Figures 1--2]{GNA17Three}). 
%: catenary (Figure~\ref{fig:cat-model}) and mammillary (Figure~\ref{fig:mamm-model}).
%We will consider models arising from these graphs (by specifying input, output, and leak sets). 
}
 \end{figure}
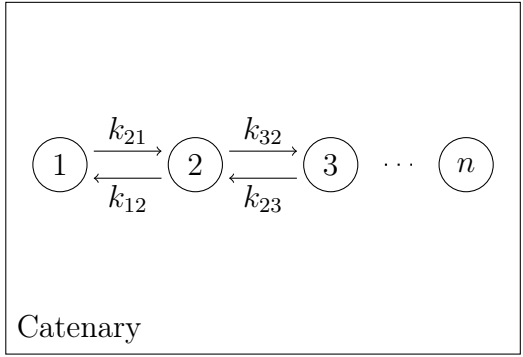
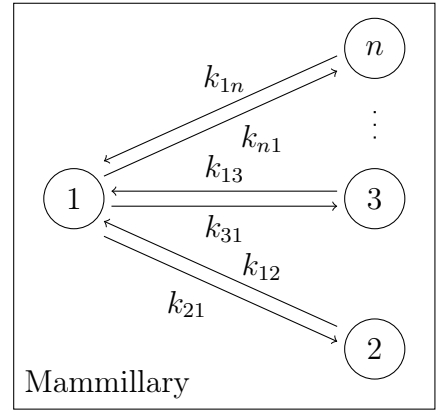

\begin{theorem}
Let $\widetilde{M}$ be a catenary, cycle, or mammillary model that has exactly one input and exactly one output, both in the first compartment, and exactly one leak. Then $\widetilde{M}$ is generically locally identifiable and so is the model M obtained by removing the leak. \label{Catenary, cycle, mam., w/wo leak, ID}
\end{theorem}

\begin{proof}
Proposition 4.7 from \cite{GHMS19One} states that the models $M$ in the statement of the theorem, with no leaks, are generically locally identifiable. Then, by Proposition~\ref{prop:add-leak-in-out}, adding a leak preserves identifiability. Thus, both $M$ and $\widetilde{M}$ are generically locally identifiable. 
\end{proof}

In other words, for every catenary, mammillary, or cycle model with input and output in the first compartment, identifiability is preserved when the leak is moved or deleted.
Other models for which identifiability is preserved when the leak is moved, are those obtained by deleting all but one leak from an ``identifiable cycle model''~\cite[Theorem 1]{MSE15Two}. 
Also, for identifiability analysis of catenary and
mammillary models where some  of the parameters are already known,
we refer the reader to~\cite{vicini-mam-cat}.

\begin{remark} \label{rmk:reln-prior-lit}
%As seen from the proof, Theorem~\ref{Catenary, cycle, mam., w/wo leak, ID} was mostly known.  
%Indeed, 
For the part of Theorem~\ref{Catenary, cycle, mam., w/wo leak, ID} concerning models with one leak, 
another approach to the proof is 
via~\cite[Section 5]{MeshkatSullivant} and~\cite[Theorem 1]{MSE15Two}, 
or (for mammillary models) through~\cite[Theorem 5.1]{van-den-Hof}.  
For the part of the theorem on catenary and mammillary models with no leaks, another proof is given in~\cite[Sections 3.1 and 4.1]{cobelli-constrained-1979}.
%The main contribution of our result, therefore, is to illuminate some infinite families for which Conjecture~\ref{conj:remove-leak} holds.
\end{remark}

We end this subsection with a result that combines two prior results.

\begin{proposition}\label{prop:id-cycle-model-0-1-leaks}
Let $M$ be a linear compartmental model obtained from an 
identifiable cycle model 
by removing all leaks 
or removing all leaks except one.
Then $M$ is generically locally identifiable.
\end{proposition}

\begin{proof}
Any such model with one leak is generically locally identifiable, by \cite[Theorem 1]{MSE15Two}. 
In particular, the model with $Leak=\{1\}$ is identifiable (and strongly connected), and so Proposition~\ref{prop:add-leak-in-out}.3 implies that 
removing the leak preserves identifiability. 
\end{proof}

\subsection{Moving Inputs and Outputs} \label{sec:cycle}

In the previous subsection, we investigated whether identifiability is preserved when a leak is moved or removed. Now we consider other operations: moving inputs and outputs. 
We show that these operations preserve identifiability in cycle models (in such models, 
%by symmetry, 
moving the input is equivalent to moving the output).  As a consequence, we obtain the main result of this subsection, Theorem \ref{Big Cycle Model theorem }, which states that every cycle model with at most one leak (and at least one input and at least one output) -- such as those considered in Examples~\ref{ex:main} and~\ref{ex:main-2} --  is identifiable. 

\begin{theorem}  \label{Big Cycle Model theorem }
Assume\footnote{This $n\geq 3$ assumption comes from the fact that the $n=1$ and $n=2$ cases reduce to catenary models.} 
$n\geq 3$. Let $M$ be an $n$-compartment cycle model with at least one input, at least one output, and at most one leak.
Then $M$ is generically locally identifiable.
\end{theorem}
\noindent 
% {\color{blue}
In the next subsection, we prove a partial converse to Theorem~\ref{Big Cycle Model theorem } 
(see Theorem~\ref{thm:partial-converse-num-leaks}). 
% }

We prove Theorem~\ref{Big Cycle Model theorem } at the end of this subsection by using the next two propositions.  For the proofs of those propositions, recall that $(A)_{ji}$ denotes the submatrix obtained by removing row $j$ and column $i$ from matrix $A$.
%since for $n=2$, the cycle model is covered by Theorem 3.1 up to relabeling input compartment?

%This theorem expands upon Theorem \ref{Catenary, cycle, mam., w/wo leak, ID} by showing that, in a cycle model, not only can the leak be moved, but also the input \textit{and} the output, and identifiability is still preserved under these operations. We prove Theorem \ref{Big Cycle Model theorem } (at the end of the subsection) by first showing that moving the output preserves identifiability. To see this, we first compute the coefficient map.

\begin{proposition} \label{prop:coeff-cycle}
Assume $n\geq 3.$ For an $n$-compartment cycle model with 
$In=\{1\}$, $Out=\{p\}$ for some $p \neq 1$, and $Leak= \emptyset$,
%no leaks, input in compartment 1 only, and output in compartment $p$ only, such that $p\neq1$ (that is, the output is not in the same compartment as the input), 
the coefficient map $c:\mathbb{R}^{n}\rightarrow \mathbb{R}^{2n-p}$ is given by: 
%$$c:\mathbb{R}^{n}\rightarrow \mathbb{R}^{2n-p}$$ 
%where 
$$(k_{21}, k_{32}, \dots , k_{1n}) \longmapsto 
    \left(e_1, ~e_2,~ \dots,~ e_{n-1}, ~
     \kappa, ~ e^*_1\kappa ,~ \dots ~,~  e_{n-p}^* \kappa  \right)~,$$
\noindent where $\kappa:= \prod_{i=2}^{p}k_{i,i-1}$, and $e_j$ 
and $e^*_j$ 
denote the 
%is the 
$j^{th}$ elementary symmetric polynomial on the sets $E=\{k_{21},\dots,k_{n,n-1}, k_{1n}\}$ %, and $e^*_j$ is the $j^{th}$ elementary symmetric polynomial on the set 
and 
$E^*=\{k_{p+2,p+1},\dots, k_{n,n-1}, k_{1n}\}$, respectively.
\end{proposition}

\begin{proof}
%Let $p$ be the output compartment, so $p\in\{2,3,\ldots,n\}$. 
In the indices, we let $n+1:=1$.  By Proposition \ref{Input-Output Equation Definition}, the input-output equation is 
\begin{equation}
    \det(\partial I-A)y_p=(-1)^{p+1}\det(\partial I-A)_{1p}u_{1}~, \label{LHS0}
\end{equation}
where $A$ is the compartmental matrix. 
The $n\times n$ matrix $A':=(\partial I-A)$ is as follows:\\
\begin{align} \label{eq:matrix-A'}
    A'~=~
        \begin{bmatrix}
        \frac{d}{dt}+k_{21} & 0& \dots & & 0 & -k_{1n}\\
        -k_{21} & \frac{d}{dt}+k_{32} & 0 & \dots & & 0 \\
        0 & \ddots & \ddots & & &  \\
        \vdots & & -k_{i, i-1} & \frac{d}{dt}+k_{i+1, i} & & \vdots\\
        & & & \ddots &\ddots & \\
        0 &  & \dots & 0 & -k_{n, n-1} & \frac{d}{dt}+k_{1n}\\
        \end{bmatrix}~.
\end{align}
We compute the determinant by expanding along the first row: 
\begin{align} \label{LHS5} %\label{LHSexpanded}
     \det{A'} ~&=~ 
        \left(\frac{d}{dt}+k_{21}\right)\det{(A')_{11}}+(-1)^{n-1}\left(-k_{1n}\right)\det{(A')_{1n}} \\
   ~&=~ \left( \frac{d}{dt}+k_{21}\right) \prod_{i=2}^n\left(\frac{d}{dt}+k_{i+1,i}\right)-k_{1n}\prod_{j=2}^{n}k_{j,j-1}
    \notag
    \\
    ~&=~
   \prod_{i=1}^n\left(\frac{d}{dt}+k_{i+1,i}\right)-\prod_{j=1}^{n}k_{j+1,j}~.    \notag
\end{align}
  Here, we used the fact that the matrix $(A')_{11}$ is lower triangular and the matrix $(A')_{1n}$ is upper triangular, so each determinant is just the products of the diagonal entries.

Next, 
we simplify equation (\ref{LHS5}) 
by recalling that the $e_j$'s denote the elementary symmetric polynomial on $E$: 
 \begin{equation*}
     \det(A') ~=~ \frac{d}{dt}^n+e_1\frac{d}{dt}^{n-1}+ \dots + e_{n-1}\frac{d}{dt}~. 
 \end{equation*}
 Thus, the left-hand side of the input-output equation~\eqref{LHS0} is: 
  \begin{equation*}
     \det(A')y_p=y^{(n)}+e_1y_p^{(n-1)}+\dots+ e_{n-1}y_p^{(1)}~, %\label{LHSfinal}
 \end{equation*}
and so the first $(n-1)$ coefficients stated in the proposition are correct. 

We now turn our attention to the remaining coefficients, those on the right-hand side of the input-output equation~(\ref{LHS0}).
It is straightforward to check, using~\eqref{eq:matrix-A'}, that the matrix $(A')_{1p}$ can be written as a block matrix:
\[(A')_{1p} %~=~
%\begin{bmatrix}
%-k_{21} & \frac{d}{dt}+k_{32} & 0 & \dots & &&0 \\
%0 & -k_{32} &&&&& \vdots \\
%\vdots && \ddots &&&& \\
%0& \dots & 0 & -k_{p,p-1} & 0 & \dots & 0 \\
%&&&&\frac{d}{dt}+k_{p+2, p+1} & &\vdots\\
%\vdots &&&&    &\ddots & 0\\
%0&\dots &&&&& \frac{d}{dt}+ k_{1n}
%\end{bmatrix}
    ~=~
\begin{bmatrix}
B & 0 \\
0 & C\\
\end{bmatrix}~,
\]
where $B$ is 
% a lower-triangular
% {\color{blue}
an upper-triangular,
% }, 
$(p-1)\times (p-1)$ matrix with diagonal entries
$-k_{21}$, $-k_{32}, \dots,$ $-k_{p,p-1}$, and 
$C$ is 
% an upper-triangular
% {\color{blue}
a lower-triangular,
% }, 
$(n-p) \times (n-p)$ matrix with the diagonal entries 
$\frac{d}{dt} + k_{p+2,p+1},$ $\dots, \frac{d}{dt} + k_{n,n-1},$ $\frac{d}{dt} + k_{1,n}$. 
Thus,
%Then, $\det{(A')_{1p}} %=\det{B}\det{(C-0B^{-1}0)} 
%=\det{B}\det{C}$. Therefore:\footnote{In the case of $p=2$, the first product in (\ref{RHS1}) becomes the empty product, that is, 1.} 
  \begin{equation}
   \det{(A')_{1p}}
    ~=~ (\det B)( \det C)
   ~=~ (-1)^{p-1} \left( \prod_{i=2}^{p}k_{i,i-1} \right) \prod_{j=p+1}^n \left( \frac{d}{dt}+k_{j+1,j} \right)~. \label{RHS1}
\end{equation}

\par Similar to before, the product of the binomials in~\eqref{RHS1} can be expressed in terms of elementary symmetric polynomials, this time on $E^*$: %=\{k_{p+2,p+1},..., k_{1n}\}$:

\begin{equation} \label{RHS3}
    \prod_{j=p+1}^n \left( \frac{d}{dt}+k_{j+1,j} \right) 
    ~=~ 
    \left( \frac{d}{dt}\right) ^{n-p} + e^*_1 \left( \frac{d}{dt}\right) ^{n-p-1} + \dots + e^*_{n-p-1} \frac{d}{dt} + e^*_{n-p}~. 
\end{equation}

%Combining (\ref{RHS3}) with (\ref{RHS1}), then (\ref{RHS1}) with (\ref{LHS0}) gives the final form of the RHS 
Using (\ref{RHS1}) with (\ref{RHS3}), we obtain the right-hand side of the input-output equation (\ref{LHS0}): 
\begin{equation}
    (-1)^{p+1}\det(\partial I-A)_{1p}u_{1} ~=~ 
    \prod_{i=2}^{p}k_{i,i-1}\left( \frac{d}{dt}^{n-p} + e^*_1 \frac{d}{dt}^{n-p-1} + \dots + e^*_{n-p}\right) u_1~. \label{IOfinal}
\end{equation}
\noindent
By inspection, the coefficients of~\eqref{IOfinal} match the last $(n-p+1)$ coefficients stated in the proposition, and so this completes the proof.
%The coefficients of the input-output equation (\ref{LHS0}) can now easily be extracted, where the $g_i$ are the coefficients from (\ref{LHSfinal}) and the $d_j$ are the coefficients from (\ref{IOfinal}): 
%\begin{align*}
%    g_i~=~ & e_{i} \qquad \qquad  ~~~~~~\text{ for } i=1,\ldots, n-1  \\
%    d_j ~=~ & \prod_{\ell=2}^{p}k_{\ell,\ell-1}e^*_{j} \qquad ~\text{for } j=0,..., n-p~.
%\end{align*}
\end{proof}

\begin{proposition}
Assume $n\geq 3$. Let M be an $n$-compartment cycle model with no leaks, exactly one input, and exactly one output. Then M is generically locally identifiable. \label{Output-Cycle IDable}
\end{proposition}

\begin{proof}
By relabeling, we may assume that $In=\{1\}$. %the input is in compartment 1.
Let $p$ denote the output compartment. If $p=1$, then the model $M$ is generically locally identifiable by Theorem \ref{Catenary, cycle, mam., w/wo leak, ID}. 

Now assume $p\neq 1$.
By Proposition \ref{prop:rank}, 
%a model is generically locally identifiable if 
we must show that 
the Jacobian matrix of the coefficient map is generically full rank. 
% By Theorem \ref{prop:coeff-cycle}, the coefficients are as follows: 
% \begin{center}
%     $(e_1, e_2,...,e_{n-1}, Le^*_0, Le^*_1,...,Le^*_{n-p})$
% \end{center}
% The total number of coefficients is $2n-p$. 
By Proposition \ref{prop:coeff-cycle}, the coefficient map $c:\mathbb{R}^n\rightarrow \mathbb{R}^{2n-p}$ is given by 
\(
    (k_{21},k_{32},\dots,k_{n,n-1},k_{1n})\longmapsto~(e_1, e_2,\dots,e_{n-1}, \kappa , e^*_1 \kappa , \dots , e^*_{n-p} \kappa),
\)
where $\kappa:=\prod_{i=2}^{p}k_{i,i-1} $.
Thus, the Jacobian matrix of $c$ is an $n\times (2n-p)$ matrix (with $2 \leq p \leq n$), and so we must show that an $(n \times n)$ submatrix has nonzero determinant.
To see this, 
it suffices to show (recall Remark~\ref{rem:alternate-proof}) 
that $e_1, e_2,\ldots,e_{n-1}, \kappa$ are algebraically independent. 
To start, the elementary symmetric polynomials $e_1, e_2, \ldots,e_{n-1}$ are algebraically independent 
(recall Remark~\ref{rem:alternate-proof}).  
Now assume for contradiction that $\kappa = 
% \prod\limits_{i=2}^{p}
\prod_{i=2}^{p}
k_{i,i-1}$ is algebraic over the field $\mathbb{F}:=\mathbb{C}(e_1,e_2,\ldots,e_{n-1})$. 
Then, by symmetry, the product of any $p-1$
distinct elements of the set
$\{k_{21},k_{32},\ldots,k_{n,n-1},k_{1n}\}$ is also algebraic over $\mathbb{F}$.
% {\color{red} why?} Since the k_ij are algebraically independent over this field. Once you have a poly over the field that kappa satisfies, can just relabel indices to get any other product of p-1 of the k_ij to satisfy that same polynomial.
By a straightforward combinatorial argument,
% EACH k_i,i-1 appears in (n-1)-choose-(p-2) such products of (p-1) elements [pick the remaining (p-2) elements from the (n-2) available k_j,j-1's]
the product
% the product of algebraic elements is itself algebraic
of all such products is $(e_n)^{\binom{n-1}{p-2}}$, so $e_n$ is algebraic over $\mathbb{F} = \mathbb{C}(e_1,e_2,\ldots,e_{n-1})$. 
This is a contradiction (see Remark~\ref{rem:alternate-proof}). %contradicts Lemma~\ref{lem:elem-sym-poly}.
% }
%Thus, $ \det \widetilde J$ is a nonzero polynomial, and t
%This completes the proof.
\end{proof}

We now prove the main result of this subsection.
\begin{proof}[Proof of Theorem \ref{Big Cycle Model theorem }]
By Proposition~\ref{Output-Cycle IDable}, the models with no leaks, exactly one input, and exactly one output are generically locally identifiable. Also, by Proposition~\ref{prop:add-leak-in-out}, adding one leak, one or more inputs, and/or one or more outputs preserves identifiability.
\end{proof}

% {\color{blue}
% NEW SUBSECTION
\subsection{Adding Leaks to a Cycle Model}\label{sec:add-leaks-to-cycle}
In the previous subsection, we saw that cycle models with one input, one output, and at most one leak are identifiable (Theorem~\ref{Big Cycle Model theorem }).  Here we present a partial converse (Theorem~\ref{thm:partial-converse-num-leaks}), which we prove at the end of the subsection.

\begin{theorem} \label{thm:partial-converse-num-leaks}
Assume $n \geq 3$.
Let $M$ be an $n$-compartment cycle model with input and output in the same compartment (and no other inputs or outputs). 
Then, $M$ is generically locally identifiable if and only if $|Leak|\leq 1$.
\end{theorem}

Our proof of Theorem~\ref{thm:partial-converse-num-leaks} relies on an
analysis of how the number and relative position of the leaks affects identifiability -- even when the input and output compartments are distinct (see Theorem~\ref{thm:cycle-many-leaks}).  We begin by computing the relevant coefficient map (Proposition~\ref{prop:coeff-map-cycle-many-leaks}), which can be viewed as the $Leak \neq \emptyset$ analogue of Proposition~\ref{prop:coeff-cycle}.

%Lastly, in this section we analyze identifiability of cycle models with more than one leak. 
%Earlier, we proved in Theorem \ref{Catenary, cycle, mam., w/wo leak, ID} that a cycle model with one or zero leaks is generically locally identifiable. 
%We continue along this line of inquiry {\color{teal} to show that both the number and relative position of leaks affects identifiability}. 

\begin{proposition} 
\label{prop:coeff-map-cycle-many-leaks}
Assume $n \geq 3$. 
Let $M$ be an $n$-compartment cycle model with $In=\{1\}$, $Out=\{p\}$ (for some $1 \leq p \leq n$), and 
$Leak = \{i_1,i_2,\dots, i_t \} \neq \emptyset$.
% $|Leak|=t$, where $2\leq t$ {\color{violet} maybe $1 \leq t$?}. 
Then the coefficient map 
$c:\mathbb{R}^{n+t}\rightarrow \mathbb{R}^{2n-p+1}$
is given by
% : 
\begin{equation*} %\label{eq:coeff-map-2+-leaks}
    (k_{21}, k_{32},\dots,k_{1,n},~
    k_{0, i_{1}}, 
    k_{0, i_{2}}, 
    \dots, 
        k_{0, i_{t}}) 
     \mapsto (e_1, e_2,\dots, e_{n-1}, e_n - \prod_{i=1}^n k_{i+1,i}, ~
     \kappa, 
     e^*_{1} \kappa, 
     \dots, 
     e^*_{n-p} \kappa) 
\end{equation*}

%\begin{equation}\label{eq:coeff-map-2+-leaks}
%    (k_{21}, k_{32},\dots,k_{i+1, i},\dots, k_{j+1, j},\dots, k_{0i},\dots, k_{0j},\dots,{\color{teal} ?}) \mapsto (e_1, e_2,\dots, e_n-F, K, e^*_{1}K, \dots, e^*_{n-p}K) 
%\end{equation}
%where $i,j \in Leak$, $F=\prod_{i=1}^n k_{i+1,i}$, 

%{\color{teal}
%\begin{equation}\label{eq:coeff-map-2+-leaks}
%\text{~OR:~}
%    (k_{i+1, i}~|~i\in [n])~\oplus~(k_{0i}~|~i\in Leak) \mapsto (e_1, e_2,\dots, e_n-F, K, e^*_{1}K, \dots, e^*_{n-p}K) 
%\end{equation}
%}
\noindent 
where 
$\kappa:= \prod_{i=2}^{p}k_{i,i-1}$, and $e_j$ 
and $e^*_j$ 
denote the 
%is the 
$j^{th}$ elementary symmetric polynomial on the sets 
$E = \{ k_{\ell+1,\ell} \mid \ell \in
% [n] 
\{1,\dots,n\}
\setminus Leak \} ~\cup~ \{ k_{\ell +1, \ell} + k_{0, \ell} \mid \ell \in Leak\}$ 
and 
$E^* = 
\{ k_{\ell+1,\ell} \mid 
p+1 \leq \ell \leq n,~ \ell \not\in Leak \} ~\cup~ 
\{ k_{\ell +1, \ell} + k_{0,\ell} \mid 
p+1 \leq \ell \leq n,~ \ell \in Leak \}$,
respectively.
\end{proposition}

\begin{proof}
The matrix $\partial I - A$ is obtained from the matrix~\eqref{eq:matrix-A'} by adding, for each leak $\ell \in Leak$, the term $k_{0 \ell}$ to the $\ell$-th diagonal entry.  The rest of the proof is now completely analogous to the proof of Proposition~\ref{prop:coeff-cycle}.
\end{proof}

It follows from Proposition~\ref{prop:coeff-map-cycle-many-leaks} 
that a cycle model with too many leaks is 
unidentifiable.
\begin{proposition} \label{prop:uniden-too-many-leaks}
Assume $n \geq 3$. 
Let $M$ be an $n$-compartment cycle model with $In=\{1\}$ and $Out=\{p\}$ (for some $1 \leq p \leq n$). 
If $|Leak| \geq n-p+2$, 
then $M$ is {\em not} generically locally identifiable.
%If $M$ is generically locally identifiable, then $|Leak| \leq n-p+1.$ 
\end{proposition}
\begin{proof}
Assume $|Leak| \geq n-p+2$.  It follows that $n + |Leak| > 2n-p+1$, 
and so 
the coefficient map $\mathbb{R}^{n + |Leak|} \to \mathbb{R}^{ 2n-p+1}$ 
(from Proposition~\ref{prop:coeff-map-cycle-many-leaks}) 
is {\em not} finite-to-one, even outside a measure-zero set.
Thus, by definition, $M$ is {\em not} generically locally identifiable.
\end{proof}

Similarly, the location of the leaks may make a cycle model
unidentifiable.

\begin{theorem} \label{thm:cycle-many-leaks}
Assume $n \geq 3$. 
Let $M$ be an $n$-compartment cycle model with $In=\{1\}$ and $Out=\{p\}$ (for some $1 \leq p \leq n$).
If there exist leaks $i,j\in Leak$ (with $i \neq j$) such that $i,j \geq p$,
then $M$ is {\em not} generically locally identifiable. 
\end{theorem}

\begin{proof}
For this model, the coefficient map 
$c:\mathbb{R}^{n+|Leak|}\rightarrow \mathbb{R}^{2n-p+1}$
is given in Proposition~\ref{prop:coeff-map-cycle-many-leaks}.  Let $J$ denote the $(2n-p+1) \times (n+|Leak|) $ Jacobian matrix of $c$. 
%It suffices to show that the columns of $J$ are linearly dependent.

Assume that there exist two leaks $i,j \geq p$.  
It follows that $\kappa= \prod_{\ell =2}^{p}k_{ \ell ,\ell -1}$
does not contain any of the following as factors: 
$k_{i+1,i}, k_{0i},k_{j+1,j}, k_{0j}$. 
Now it is straightforward to check from the coefficient map that, for each coefficient 
$c_{\ell}$, except for the coefficient 
$e_n - \prod_{i=1}^n k_{i+1,i}$, the partial derivatives satisfy the following equalities:
\begin{align*}
    \frac{\partial c_{\ell}}{\partial k_{i+1,i}} ~=~ 
            \frac{\partial c_{\ell}}{\partial k_{0i}}
    \quad {\rm and} \quad
    \frac{\partial c_{\ell}}{\partial k_{j+1,j}} ~=~ 
            \frac{\partial c_{\ell}}{\partial k_{0j}}~.
\end{align*}

Thus, the columns of $J$ corresponding to $k_{i+1,i}$ and $k_{0i}$, which we denote by $C_{i+1,i}$ and $C_{0i}$, respectively, are identical in every coordinate, except the one corresponding the coefficient $e_n - \prod_{i=1}^n k_{i+1,i}$.  Thus, the vector $C_{i+1,i} - C_{0i}$ 
has 0 in all but one coordinate.  By the same argument,  $C_{j+1,j} - C_{0j}$ is 0 in all but the same coordinate.  

Thus, the four columns 
$C_{i+1,i}, C_{0i},C_{j+1,j}, C_{0j}$ of the Jacobian matrix $J$ are linearly dependent.  We conclude that $J$ is {\em not} full rank, and so the model $M$ is {\em not} generically locally identifiable.
\end{proof}

We end this subsection by analyzing two cases when identifiability of a cycle model is characterized by the number of leaks.

\begin{corollary} \label{cor:in-out-distance-1}
Assume $n \geq 3$. 
Let $M$ be an $n$-compartment cycle model with $Input=\{i\}$ and $Output=\{i-1 \mod n\}$ (for some $1 \leq i \leq n$). 
Then $M$ is generically locally identifiable if and only if $|Leak|\leq 1.$
\end{corollary}
\begin{proof}
The backward direction ($\Leftarrow$) follows from Theorem~\ref{Big Cycle Model theorem }.
For the forward direction ($\Rightarrow$), we can relabel compartments so that $In=\{1\}$ and $Out =\{n\}$, 
so, by Proposition~\ref{prop:uniden-too-many-leaks}, $|Leak| \leq 1.$
\end{proof}

\begin{proof}[Proof of Theorem~\ref{thm:partial-converse-num-leaks}]
One direction ($\Leftarrow$) is immediate from Theorem~\ref{Big Cycle Model theorem }.
For the other direction ($\Rightarrow$), by relabeling, we may assume that $In=Out=\{1\}$, and so the desired result follows directly from Theorem~\ref{thm:cycle-many-leaks}.
\end{proof}
% }

\subsection{Adding Incoming and  Outgoing Edges to Cycle Models} \label{sec:new-edges}
In this section, we introduce a new class of linear compartment models, which can be viewed as a hybrid between cycle and mammillary models, as they are constructed by adding certain ``incoming'' or ``outgoing'' edges 
(like those in the mammillary model) to a cycle model. 

We prove that 
when all incoming edges or all outgoing edges are added, and there is at most one leak, 
the resulting model is identifiable
(Theorem~\ref{thm:fin-wing-0-1-leak-id}).
Afterward, we consider the case of adding only a subset of the incoming edges or outgoing edges.  
Specifically, we compute the coefficient maps (Propositions~\ref{prop:coeff-fin} and~\ref{prop:coeff-wing}) 
and assess identifiability (Theorem~\ref{thm:fin-wing-partial}).

In what follows, we refer to an edge $j\rightarrow i$ by its edge-label parameter $k_{ij}$.
%will use the parameter $k_{ij}$ to refer to the edge $j\rightarrow i$. 

\begin{definition}\label{def:incoming-outgoing-edges}
Consider an $n$-compartment model. 
\begin{packed_enum}
    \item An \emph{incoming edge} is an edge from compartment $i$ \emph{to} compartment 1, where $i\in \{2,3,\dots,n-1\}$. The set of all incoming edges is 
$\{k_{12},k_{13},\dots,k_{1,n-1}\}$.
    \item An \emph{outgoing edge} is an edge \emph{from} compartment 1 to compartment $j$, where $j\in \{3,4,\dots,n\}$. The set of all outgoing edges is 
$\{k_{31},k_{41},\dots,k_{n1}\}$.
%    \item An \emph{interior edge} is either an incoming edge or an outgoing edge. 
    %The set of all \emph{interior edges} is $\{k_{12},k_{13},k_{14},\dots,k_{1,n-1};~k_{31},k_{41},\dots,k_{n1}\}$.
\end{packed_enum}
\end{definition}

\emph{Fin} and {\em Wing} models are obtained from cycle models 
%with $In=Out=\{1\}$ and $Leak=\emptyset$, 
by adding all incoming (respectively, outgoing) edges; see Figure~\ref{fig:fin-and-wingmodel} 
for the underlying graphs, denoted by ${\rm Fin}_n$ and ${\rm Wing}_n$.

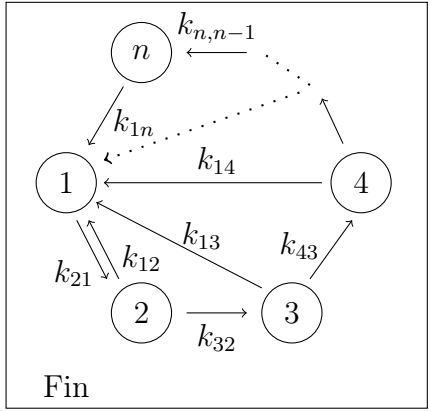
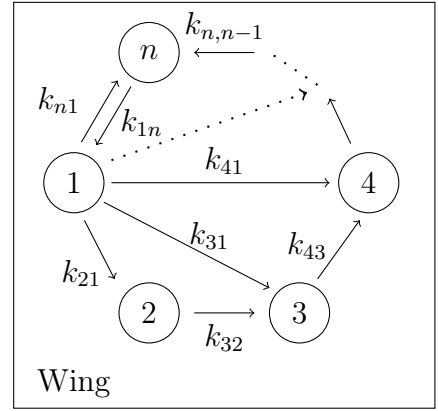
\begin{figure}[ht!]
\subfigure[The Fin graph (${\rm Fin}_n$)]{\label{fig:fin-model}
    \centering
	\begin{tikzpicture}[scale=2]
 %-----------------------------
 % FIN 
 %-----------------------------
  	\draw (-1,0) circle (0.2);	
 % 	\draw (1,0) circle (0.2);	
  	\draw (-0.5,0.866) circle (0.2);	
  	\draw (0.5,-0.866) circle (0.2);	
  	\draw (-0.5,-0.866) circle (0.2);
  	\draw (.96, 0) circle (0.2);
 %	\draw (0.5,0.866) circle (0.2);	
 % ARROWS%
    \draw[->] (-0.93, -0.25) -- (-0.73, -0.64 );
  	 \draw[->]  (-0.2,-0.866) -- (0.2,-0.866) ;
 	 \draw[->]  (0.62, -0.64) -- (0.9,-0.25); 
 	 \draw[loosely dotted,thick] (0.62,0.65) -- (0.3,0.866);
 	 \draw[->]  (0.2,0.866) -- (-0.2,0.866);
 	 \draw[->] (-0.62, 0.64 ) -- (-0.85, 0.25);
 	 \draw[->] (.84, .26) -- (.7, .56);  
 	 \draw[loosely dotted,thick, ->] (0.55, 0.6) -- (-.75, .15);
 %	 \draw[->]  (-0.5,0.866)-- (-1,0);
 %	 \draw[->] (-0.8, 0) -- (-0.7, -0.9);	
 %incoming ARROWS%
    \draw[->] (-.65, -0.64) -- (-0.85, -.25);
    \draw[->] (0.3, -.7) -- (-.8, -.13);
    \draw[->] (.7, 0) -- (-.75, 0);
 % RATES
   	 \node[] at (0.0, -1.03) {$k_{32}$};
   	 \node[] at (0.0, 1.03) {$k_{n,n-1}$};
   	 \node[] at (-0.95, -0.6) {$k_{21}$};
   	 \node[] at (-0.55, 0.4) {$k_{1n}$};
   	 \node[] at (0.55, -0.42) {$k_{43}$};
   	 \node[] at (0, .15) {$k_{14}$};
   	 \node[] at (-.1, -.35) {$k_{13}$};
   	 \node[] at (-.5, -.5) {$k_{12}$};
 % LABEL: numbers
     	\node[] at (-1, 0) {1};
     	\node[] at  (-0.5,-0.866) {$2$};
     	\node[] at  (+0.5,-0.866) {$3$};
     	\node[] at (-0.5,0.866) {$n$};
     	\node[] at (.96, 0) {$4$};
%  %OUTPUT
%   	\draw (-1.33,-.49) circle (0.05);	
%  	 \draw[-] (-1, -.2 ) -- (-1.3, -.45);	
%  	 \node[] at (-1.65, -.49) {out};
%  % Input
%  	 \draw[->] (-1.7, 0) -- (-1.3, 0);	
%   	 \node[] at (-1.9, 0) {in};
 % LABEL: Model name
    %  	\node[above] at (-1.8, -1.5) {Fin};
     	\node[above] at (-1, -1.5) {Fin};
  % BOX
%  \draw (-2.2,-1.5) rectangle (1.4, 1.2);
    \draw (-1.4,-1.5) rectangle (1.4, 1.2);
 \end{tikzpicture}  
 }
\hfill %
\subfigure[The Wing graph (${\rm Wing}_n$)]{\label{fig:wing-model}
    \centering
	\begin{tikzpicture}[scale=2]
 %-----------------------------
 % WING 
 %-----------------------------
  	\draw (-1,0) circle (0.2);	
 % 	\draw (1,0) circle (0.2);	
  	\draw (-0.5,0.866) circle (0.2);	
  	\draw (0.5,-0.866) circle (0.2);	
  	\draw (-0.5,-0.866) circle (0.2);
  	\draw (.96, 0) circle (0.2);
 %	\draw (0.5,0.866) circle (0.2);	
 % ARROWS%
    \draw[->] (-0.93, -0.25) -- (-0.73, -0.64 );
  	 \draw[->]  (-0.2,-0.866) -- (0.2,-0.866) ;
 	 \draw[->]  (0.62, -0.64) -- (0.9,-0.25); 
 	 \draw[loosely dotted,thick] (0.62,0.65) -- (0.3,0.866);
 	 \draw[->]  (0.2,0.866) -- (-0.2,0.866);
 	 \draw[->] (-0.62, 0.64 ) -- (-0.85, 0.25);
 	 \draw[->] (.84, .26) -- (.7, .56);  
 	 \draw[loosely dotted,thick, <-] (0.55, 0.6) -- (-.75, .15);
 %	 \draw[->]  (-0.5,0.866)-- (-1,0);
 %	 \draw[->] (-0.8, 0) -- (-0.7, -0.9);	
 %incoming ARROWS%
    \draw[<-] (-.71, 0.68) -- (-0.95, .27);
    \draw[<-] (0.3, -.7) -- (-.8, -.13);
    \draw[<-] (.7, 0) -- (-.75, 0);
 % RATES
   	 \node[] at (0.0, -1.03) {$k_{32}$};
   	 \node[] at (0.0, 1.03) {$k_{n,n-1}$};
   	 \node[] at (-0.95, -0.6) {$k_{21}$};
   	 \node[] at (-0.55, 0.4) {$k_{1n}$};
   	 \node[] at (0.55, -0.42) {$k_{43}$};
   	 \node[] at (0, .15) {$k_{41}$};
   	 \node[] at (-.1, -.35) {$k_{31}$};
   	 \node[] at (-1.1, 0.55) {$k_{n1}$};
 % LABEL: numbers
     	\node[] at (-1, 0) {1};
     	\node[] at  (-0.5,-0.866) {$2$};
     	\node[] at  (+0.5,-0.866) {$3$};
     	\node[] at (-0.5,0.866) {$n$};
     	\node[] at (.96, 0) {$4$};
%  %OUTPUT
%   	\draw (-1.33,-.49) circle (0.05);	
%  	 \draw[-] (-1, -.2 ) -- (-1.3, -.45);	
%  	 \node[] at (-1.65, -.49) {out};
%  % Input
%  	 \draw[->] (-1.7, 0) -- (-1.3, 0);	
%   	 \node[] at (-1.9, 0) {in};
 % LABEL: Model name
    %  	\node[above] at (-1.8, -1.5) {Wing};
     	\node[above] at (-1, -1.5) {Wing};
  % BOX
%  \draw (-2.2,-1.5) rectangle (1.4, 1.2);
 \draw (-1.4,-1.5) rectangle (1.4, 1.2);
 \end{tikzpicture}
 }
    \caption{The Fin and Wing graphs with $n$ compartments.}
%    \caption{The $n$-compartment Fin model (Figure~\ref{fig:fin-model}) is obtained from the cycle model by adding all \emph{incoming} edges. The $n$-compartment Wing model (Figure~\ref{fig:wing-model}) is obtained from the cycle model by adding all \emph{outgoing} edges.}
    \label{fig:fin-and-wingmodel}
\end{figure}

% \begin{conjecture}[Fin and Wing are identifiable] \label{conj:fin-wing}
% For $n \geq 3$, the $n$-compartment Fin and Wing models (Figure~\ref{fig:fin-and-wingmodel}) are generically locally identifiable.
% \end{conjecture}

\begin{theorem}\label{thm:fin-wing-0-1-leak-id}
Let $n \geq 3$, and let $M$ be an $n$-compartment 
Fin model  
$({\rm Fin}_n, In, Out, Leak)$ 
or Wing model 
$({\rm Wing}_n, In, Out, Leak)$, 
with $In=Out=\{1\}$ and at most one leak.
%{\color{teal}Wait, but the way we define Fin/Wing models following Def~\ref{def:incoming-outgoing-edges} is by taking In=Out=\{1\} and no leaks} 
Then $M$ is generically locally identifiable.
\end{theorem}
\begin{proof}
It is straightforward to check that, 
when all leaks are added,
the Fin and Wing models are identifiable cycle models  (Definition~\ref{def:id-cycle-model}).
In particular, 
condition~(4) is satisfied 
because 
the graphs ${\rm Fin}_n$ and ${\rm Wing}_n$ 
satisfy
$|E|=2|V|-2$ and
are inductively strongly connected with respect to compartment 1 and the orderings $\{1, 2, \dots, n\}$ and $\{1, n, n-1, \dots, 2\}$, respectively~\cite[Remark 1]{MSE15Two}. 
%{\color{violet} Cite a result (state earlier in paper?) that says that inductively strongly connected implies identifiable cycle model?} 
%{\color{teal}Actually ``Inductively strongly connected \textit{and} $|E|=2|V|-2$ implies identifiable cycle model'', I think.}
So, by Proposition~\ref{prop:id-cycle-model-0-1-leaks}, $M$ is generically locally identifiable.
\end{proof}

When only \textit{some} incoming (respectively, outgoing) edges are added, the resulting submodel of a Fin (respectively, Wing) model fails to be inductively strongly connected. 
Thus, to check whether such a model is an identifiable cycle model (Definition~\ref{def:id-cycle-model}), we must compute and analyze the coefficient map.
%{\color{red}
%In other words, these submodels are not identifiable cycle models {\color{violet}or at least it's not easy to check if the model is an identifiable cycle model?} and so we cannot apply Proposition~\ref{prop:id-cycle-model-0-1-leaks} to show identifiability. 
%{\color{teal}How about: ``In other words, for these submodels, Definition~\ref{def:id-cycle-model}.4 is not straightforward to check if they are identifiable cycle models without computing the coefficient maps.''}
Accordingly, we next compute the coefficient maps of Fin and Wing models (Propositions~\ref{prop:coeff-fin} and \ref{prop:coeff-wing}) -- the coefficient maps for submodels are then obtained by setting some rate constants to 0 -- and then prove identifiability of submodels of Fin and Wing with only one or two incoming edges or outgoing edges (Theorem~\ref{thm:fin-wing-partial}).

%\subsubsection{Identifiability of Fin and Wing Models}
We begin by computing the coefficient maps for Fin and Wing models. The proofs are similar, but the differences are subtle enough that we must prove them separately. 

\begin{proposition}[Coefficient map for Fin] \label{prop:coeff-fin}
Assume $n\geq 3$. 
For $2 \leq \ell \leq n-1$, define the following expressions: 
\begin{packed_enum}
    \item 
    % {\color{blue} 
    $e_m^{[\ell]}$
    % } 
    is the $m^{\text{th}}$ elementary symmetric polynomial\footnote{By convention, 
    % {\color{blue} 
    $e^{[\ell]}_0:=1$.
    % }.
    } on 
$
% {\color{blue} 
E^{[\ell]}
% } 
:=\{ k_{1\ell}+k_{\ell+1,\ell},~k_{1,\ell+1}+k_{\ell+2,\ell+1},~ \ldots ~, ~k_{1,{n-1}}+k_{n,n-1}, ~k_{1n}\}$, 
    \item $P_{\ell}:=k_{1{\ell}}(k_{21}k_{32}\cdots k_{{\ell}, {\ell}-1})$, and
    \item 
%        \begin{align*}
    $\phi_{\ell} := 
    % {\color{blue} 
    e_{\ell} ^{[2]}
    % } 
    +k_{21} 
    % {\color{blue} 
    e^{[2]}_{\ell - 1}
    % } 
    - \sum \limits_{i=2}^{{\ell}} P_i
    % {\color{blue} 
    e_{{\ell}-i}^{[i+1]}$.
    % }$.
%    \end{align*}
\end{packed_enum}
Then, for the $n$-compartment model 
$({\rm Fin}_n, In, Out, Leak)$ 
with $In=Out=\{1\}$ and $Leak=\emptyset$,
%{\color{red} Fin model in Figure~\ref{fig:fin-model}}, 
the coefficient map $c:\mathbb{R}^{2n-2}\rightarrow \mathbb{R}^{2n-2}$ is given by 
\begin{align} \label{eq:coeff-map-fin}
     (k_{21},k_{32},k_{43}, \dots, k_{1n},~ k_{12}, k_{13}, \dots ,k_{1,n-1}) 
    ~ \mapsto ~
    %RHS coeffs
    \bigg( 
    % {\color{blue} 
    e_1^{[2]},~ e_2^{[2]},~ \ldots, ~ e_{n-1}^{[2]},
    %LHS coeffs
    ~e_1^{[2]} 
    % } 
    +k_{21}~,
    \phi_2,~\phi_3, \dots, ~\phi_{n-1}
    \bigg)~.
\end{align}
\end{proposition}

\begin{proof}
The matrix $A':=(\partial I - A)$ is 
\begin{center}
    $A'=\begin{bmatrix}
     \frac{d}{dt}+k_{21} & -k_{12} & -k_{13} & \dots & -k_{1,n-1} & -k_{1n}\\
    -k_{21} & \frac{d}{dt}+(k_{12}+k_{32}) & 0 &\dots &  &0\\
    0 & -k_{32} & \frac{d}{dt}+(k_{13}+k_{43})  \\
    0 & 0 & -k_{43} & \ddots&  & \vdots\\
    \vdots & \vdots & &  & \frac{d}{dt}+(k_{1,n-1}+k_{n,n-1}) & 0\\
    0 & 0 & \dots & 0 & -k_{n,n-1} & \frac{d}{dt}+k_{1n}\\ 
    \end{bmatrix}$. 
\end{center}

\noindent The matrix $(A')_{11}$ is lower triangular, with diagonal entries $(d/dt + k_{12}+k_{32}),~(d/dt + k_{13}+k_{43}),\ldots,~(d/dt + k_{1,n-1}+k_{n,n-1}),~(d/dt + k_{1n})$. Hence, the non-monic coefficients of the right-hand side of the input-output equation~(\ref{eq:i-o}) are the elementary symmetric polynomials 
$
% {\color{blue} 
e_1^{[2]}, ~ e_2^{[2]}, \ldots, ~ e_{n-1}^{[2]} 
% }
$ on 
$
% {\color{blue}  
E^{[2]} 
% } 
=\{k_{12}+k_{32},\ldots,~k_{1,{n-1}}+k_{n,n-1},~k_{1n}\}$, 
and these are the first $n-1$ coefficients in~(\ref{eq:coeff-map-fin}).

To obtain the left-hand side of~(\ref{eq:i-o}), we expand along the first row of $A'$:
\begin{align} \label{eq:LHS-fin}
    \det A' ~&=~ \left( \frac{d}{dt} + k_{21} \right) \det (A')_{11} 
    % \\
    \notag
        % & \quad \quad \quad 
        + k_{12} \det (A')_{12} 
        - k_{13} \det (A')_{13} 
        + \dots 
        + (-1)^{n} k_{1n} \det (A')_{1n} 
    \\ 
    % \notag
    ~&=~ \left( \frac{d}{dt} + k_{21} \right) \left( \left(\frac{d}{dt} \right) ^{n-1} + \left(\frac{d}{dt}\right)^{n-2} 
    % {\color{blue} 
    e_1^{[2]} 
    % } 
    + \dots + 
    % {\color{blue} 
    e_{n-1}^{[2]} 
    % }
    \right) + \sum_{j=2}^n (-1)^{j} k_{1j} \det (A')_{1j}~.
\end{align}

\noindent 
Letting $K_j:=k_{1,j}+k_{j+1,j}$ (for $2 \leq j \leq n-1$) and $K_n:=k_{1n}$,
we obtain: 
% %without subsitution K_j
% \begin{align*} 
% % (A')_{1j} ~=~ 
% \left[
% \begin{array}{cccc|cccc}
% -k_{21} & \frac{d}{dt}+k_{12}+k_{32} & & & 0 & \dots & & 0\\
% 0 & -k_{32} & \ddots & & & & \\
% \vdots &  & \ddots & \frac{d}{dt}+(k_{1,j-1}+k_{j,j-1}) & & & \\
% 0 & \dots &  & -k_{j,j-1} & 0 & \dots & & 0\\ \hline
% 0 & \dots & & 0 & \frac{d}{dt}+(k_{1,j+1}+k_{j+2,j+1}) & 0 & \dots & 0\\ 
% \vdots &  & & \vdots & -k_{j+2,j+1} &  &  & \vdots\\
% & & & & &\ddots & \ddots & 0 \\
% 0 &  \dots & & 0    `&  & & -k_{n,n-1} & \frac{d}{dt}+k_{1n}\\ 
% \end{array}
% \right]
% \end{align*}

% %with subsitution K_j
\begin{align*} 
(A')_{1j} ~=~ 
\left[
\begin{array}{cccc|cccc}
-k_{21} & \frac{d}{dt}+K_2 & & & 0 & \dots & & 0\\
0 & -k_{32} & \ddots & & \vdots & & & \vdots \\
\vdots &  & \ddots & \frac{d}{dt}+K_{j-1} & & & \\
0 & \dots &  & -k_{j,j-1} & 0 & \dots & & 0\\ \hline
0 & \dots & & 0 & \frac{d}{dt}+K_{j+1} & 0 & \dots & 0\\ 
\vdots &  & & \vdots & -k_{j+2,j+1} & \ddots &  & \vdots\\
& & & & &\ddots & \frac{d}{dt}+K_{n-1} & 0 \\
0 &  \dots & & 0 &  & & -k_{n,n-1} & \frac{d}{dt}+K_n\\ 
\end{array}
\right]~.
\end{align*}

We see that $(A')_{1j}$ is block diagonal, and both blocks are triangular, so we have:
%We see that $\det{(A')_{1j}}$ is the product of the determinants of two block matrices -- the $(j-1)\times (j-1)$ upper-triangular matrix
%and the $(n-j)\times (n-j)$ lower-triangular matrix:
\begin{align*}
    \det{(A')_{1j}}~&=~
    (-1)^{j-1}(k_{21}k_{32}\cdots k_{j,j-1}) \cdot \left(\frac{d}{dt}+K_{j+1}\right)\left(\frac{d}{dt}+K_{j+2}\right) \cdots % \left(\frac{d}{dt}+K_{n-1}\right) 
    \left(\frac{d}{dt}+K_n\right)
    \\ \notag
    ~&=~ (-1)^{j-1}\frac{P_j}{k_{1j}}\left(\left(\frac{d}{dt}\right)^{n-j} + \left(\frac{d}{dt}\right)^{n-j-1} 
    % {\color{blue} 
    e_1^{[j+1]}
    % } 
    + \ldots + 
    % {\color{blue} 
    e_{n-j}^{[j+1]}
    % } 
    \right)~. 
\end{align*}

Using the above expression, we compute a sum from~(\ref{eq:LHS-fin}):
\begin{align*}
    \sum_{j=2}^n (-1)^{j} k_{1j} \det (A')_{1j} ~&=~
    - \sum_{j=2}^n P_j \left(\left(\frac{d}{dt}\right)^{n-j} + \left(\frac{d}{dt}\right)^{n-j-1} 
    % {\color{blue} 
    e_1^{[ j+1]}
    % } 
    + \ldots + 
    % {\color{blue} 
    e_{n-j}^{[j+1]} 
    % } 
    \right)\\
    ~&=~
    -\sum\limits_{j=2}^n \left(\frac{d}{dt}\right)^{n-j} \sum\limits_{i=2}^{j} P_i 
    % {\color{blue} 
    e^{[i+1]}_{j-i}
    % } 
    ~.
\end{align*}

By substituting the above expression into equation~\eqref{eq:LHS-fin}, and collecting coefficients of powers of $\frac{d}{dt}$, it is straightforward to verify that the non-monic coefficients of the left-hand side of the input-output equation~\eqref{eq:i-o} match the final $n-1$ coordinates in~\eqref{eq:coeff-map-fin}. 
\end{proof}

\begin{proposition}[Coefficient map for Wing] \label{prop:coeff-wing}
Assume $n\geq 3$. 
For the $n$-compartment 
model 
$({\rm Wing}_n, In, Out, Leak)$ 
with $In=Out=\{1\}$ and $Leak=\emptyset$,
%{\color{red} Wing model in Figure~\ref{fig:wing-model}}, 
the coefficient map $c:\mathbb{R}^{2n-2} \to \mathbb{R}^{2n-2}$ 
is given by 
\begin{align} \label{eq:coeff-map-wing}
    & (k_{32}, k_{43}, \dots , k_{1n}, k_{21}, k_{31},\dots,k_{n1}) \\
    & \quad \quad \mapsto 
    \big( e'_1,e'_2, \dots ,e_{n-1}', ~ 
        e'_1+K,~
        e'_2+e'_1K, ~
\psi_3,~\psi_4, ~\dots, ~\psi_{n-2},~
%        e_3'+e_{2}'K + Q_n ,~
%        \dots , \\
%        \notag
%    & \quad \quad \quad \quad \quad 
%        e_j'+e_{j-1}'K   - \sum_{i=n-j+2}^n Q_i h^i_{i-n+j-2}  , ~
%        \dots, ~
% 0 + 
e_{n-1}'K  - \sum_{i=2}^n Q_i h^i_{i-2}  \big)~,     \notag
\end{align}
where, for $3 \leq j \leq n-2$,
\[
\psi_j ~:=~   e_j'+e_{j-1}'K   - \sum_{i=n-j+2}^n Q_i h^i_{i-n+j-2} ~ , ~
\]
and $e'_m$ and $h^j_m$ are the $m^{th}$ elementary symmetric polynomials\footnote{By convention, $h^j_0:=1$} on $E':=\{k_{32}, k_{43},\dots,k_{n,n-1}, k_{1n}\}$ and 
$H^j := \{k_{32}, k_{43},\dots, k_{j,j-1}\}$ (for $3 \leq j \leq n$), 
respectively, and 
$Q_j := k_{1n} k_{j1}  (k_{j+1,j}k_{j+2,j+1} \dots k_{n,n-1}) $ (for $2 \leq j \leq n$)
and $K:=k_{21}+k_{31}+\dots+k_{n1}$. %\sum_{i=2}^nk_{i,1}$.
\end{proposition}

\begin{proof}
For this model, the matrix $A':=(\partial I-A)$ is as follows:
\begin{center}
    $A'=\begin{bmatrix}
     \frac{d}{dt}+k_{21}+k_{31}+\dots+k_{n1} & 0 & \dots & & 0 & -k_{1n}\\
    -k_{21} & \frac{d}{dt}+k_{32} & 0 &\dots &  &0\\
    -k_{31} & -k_{32} & \frac{d}{dt}+k_{43}  \\
    -k_{41} & 0 & -k_{43} & &  & \vdots\\
    \vdots & \vdots & & \ddots & \ddots & 0\\
    -k_{n1} & 0 & \dots & 0 & -k_{n,n-1} & \frac{d}{dt}+k_{1n}\\ 
    \end{bmatrix}$. 
\end{center}
The matrix $(A')_{11}$ is lower triangular, with diagonal entries $(d/dt+k_{32})$, $(d/dt+k_{43})$, \dots, $(d/dt+k_{1n})$.  Hence, the non-monic coefficients of the right-hand side of the input-output equation~\eqref{eq:i-o} are the elementary symmetric polynomials $e'_{1}$, $e'_{2}$, \dots, $e'_{n-1}$ on $E':=\{k_{32}, k_{43},\dots,k_{n,n-1}, k_{1n}\}$, which match the first $n-1$ coordinates in~\eqref{eq:coeff-map-wing}. 

Next, to obtain the left-hand side of~\eqref{eq:i-o}, we expand along the first column of $A'$:
\begin{align} \label{eq:LHS-wing}
    \det A' ~&=~ \left( \frac{d}{dt} + k_{21}+ k_{31}+ \dots + k_{n1} \right) \det (A')_{11} \\
    \notag
        & \quad \quad \quad 
        + k_{21} \det (A')_{21} 
        - k_{31} \det (A')_{31} 
        + \dots 
        + (-1)^{n} k_{n1} \det (A')_{n1} 
    \\ 
    \notag
    ~&=~ \left( \frac{d}{dt} + K \right) \left( \left(\frac{d}{dt} \right) ^{n-1} + \left(\frac{d}{dt}\right)^{n-2}e'_1 + \dots + e'_{n-1} \right) + \sum_{j=2}^n (-1)^{j} k_{j1} \det (A')_{j1}~.
\end{align}

Next, we compute:
\begin{align*} 
(A')_{j1} ~=~ 
\left[
\begin{array}{cccc|ccc|c}
    0 & \dots& & 0& 0&\dots & 0 &  -k_{1n}\\ 
    \hline
    \frac{d}{dt}+k_{32} & 0 & \dots &0 &0&\dots & 0 &0 \\
    -k_{32} & \ddots & & \vdots &\vdots && \vdots & \vdots \\
    & \ddots && 0 & && &  \\
    0 &  & -k_{j-1,j-2} & \frac{d}{dt}+k_{j, j-1} & 0 & \dots & 0 & 0 \\
    \hline
    0 & \dots & & 0 & -k_{j+1, j} & \frac{d}{dt}+k_{j+2, j+1}&  & 0\\
    \vdots &&&\vdots  & & \ddots &\ddots   \\
    0 & \dots & & 0 & & & -k_{n,n-1} & \frac{d}{dt}+k_{1n}\\
\end{array}
\right]
\end{align*}

We see that $\det (A')_{j1}$ is the product of the determinants of three block matrices: the $(1 \times 1)$ upper-right matrix, % $(-k_{1n})$, 
a $(j-2)\times (j-2)$ lower-triangular matrix with diagonal entries $(d/dt + k_{i,i-1})$, and an $(n-j) \times (n-j)$ upper-triangular matrix with diagonal entries $-k_{i+1,i}$:
\begin{align*}
\det (A')_{j1}
        ~&=~ (-1)^{j+1} k_{1n} \left( k_{j+1, j} k_{j+2,j+1}\dots k_{n,n-1}\right) 
        \left( \frac{d}{dt} + k_{32} \right) 
        \left( \frac{d}{dt} + k_{43} \right) 
        \dots
        \left( \frac{d}{dt} + k_{j,j-1} \right) \\
        ~&=~
        (-1)^{j+1} \frac{Q_j}{k_{j1}} \left( 
            \left(\frac{d}{dt}\right)^{j-2} + 
            h^j_{1} \left(\frac{d}{dt}\right)^{j-3} + 
            \dots +             
            h^j_{j-2} 
            \right)~.
\end{align*}

Thus, we can compute a sum from~\eqref{eq:LHS-wing}:
\begin{align*}
 \sum_{j=2}^n (-1)^{j} k_{j1} \det (A')_{j1}
    ~&=~
    - \sum_{j=2}^n  {Q_j} \left( 
            \left(\frac{d}{dt}\right)^{j-2} + 
            h^j_{1} \left(\frac{d}{dt}\right)^{j-3} + 
            \dots +             
            h^j_{j-2} 
            \right)~
            \\
    ~&=~     
    - \sum_{j={2}}^{n} \left( \frac{d}{dt} \right)^{n-j} \sum_{i=n-j+2}^n Q_i h^i_{i-n+j-2}~.
\end{align*}
% {\color{red} I'm not seeing this last expression - it would imply that there's no $\left(\frac{d}{dt}\right)^{0}$ term in the expansion, right? But there is. 
% {\color{violet}  Good catch! I edited above... ok now??} {\color{blue} Happy now! I will delete the colors.}
% I have:
% \begin{align*}
%     ~&=~     
%     - \sum_{j=2}^{n} \left( \frac{d}{dt} \right)^{j-2} \sum_{i=j}^n Q_i h^i_{i-j}~.
% \end{align*}
% }

By substituting the above expression into equation~\eqref{eq:LHS-wing}, and collecting coefficients of powers of $\frac{d}{dt}$, it is straightforward to verify that the coefficients of the left-hand side of the input-output equation~\eqref{eq:i-o} match the final $n-1$ coordinates in~\eqref{eq:coeff-map-wing}. 
\end{proof}

%{\color{red} A useful computation for elementary symmetric polynomials like those in $E^l$ in theorem 3.11 for when we want to construct the Jacobian: 
%
%\par Let $X=:\{x_1, x_2, .., x_n\}$ be a set of $n$ variables, where $x_i:= (a_i+b_i)$. 
%
%For $1 \leq m \leq n$, the $m$-th elementary symmetric polynomial on $X$ is as follows: \\
%\begin{center}
%    $e_m ~=~ 
%    \sum\limits_{j_{1}<j_{2}<\dots<j_m} x_{j_1}...x_{j_m}
%    ~=~
%    \sum\limits_{\substack{j_2<...<j_m \\ j_k \neq i}} (a_i+b_i) (x_{j_2}...x_{j_m}) + 
%    \sum\limits_{\substack{ l_1<...<l_m \\ l_k \neq i} } x_{l_1}...x_{l_m} 
%~=~
%    \sum\limits_{\substack{j_2<...<j_m \\ j_k \neq i}}  a_i(x_{j_2}...x_{j_m}) + \sum\limits_{\substack{j_2<...<j_m \\ j_k \neq i}}b_i (x_{j_2}...x_{j_m}) +
%    \sum\limits_{\substack{ l_1<...<l_m \\ l_k \neq i} } x_{l_1}...x_{l_m} $~,
%    
%\end{center}
%for any $1 \leq i \leq n$.
%Thus, taking the partial derivative with respect to $a_i$ yields: 
%\begin{equation}
%   \dfrac{\partial e_m}{\partial a_i}~=~ \sum\limits_{
%        \substack{j_2<...<j_m \\ j_k \neq i}        
%        } x_{j_2}...x_{j_m}~=~e_{m-1}\{\hat{x_i}\}~=~\dfrac{\partial e_m}{\partial b_i},
%\end{equation}
%where $e_{m-1}\{\hat{x_i}\}$ is the $(m-1)$-th elementary symmetric polynomial on the set $X 
%\smallsetminus \{x_i\}$
%}

\begin{remark} \label{rmk:coeff}
In~\cite[Theorem~4.5]{GNA17Three}, the authors give a formula for the coefficient map of any strongly connected model with $In=Out=\{1\}$ and \emph{at least one leak}.
This formula is in terms of spanning, incoming forests of subgraphs of the underlying graph $G$. This formula agrees with 
the coefficient maps we compute in Propositions~\ref{prop:coeff-cycle},~\ref{prop:coeff-fin}, and~\ref{prop:coeff-wing}, even though those models have \emph{no leaks} (where, in the notation of~\cite{GNA17Three}, we take $\widetilde{G}=G$). This suggests that the result \cite[Theorem~4.5]{GNA17Three} may generalize to models without leaks.
\end{remark}

The coefficient maps of the Fin and Wing models (Proposition~\ref{prop:coeff-fin} and~\ref{prop:coeff-wing}) are complicated, so analyzing the resulting Jacobian matrices (with the aim of assessing identifiability of submodels) is difficult.  We therefore only conjecture that these submodels are identifiable.

\begin{conjecture}[{Submodels of} Fin and Wing are identifiable] \label{conj:fin-wing}
For $n \geq 3$, 
let $G$ be a graph obtained from the $n$-compartment cycle graph
in Figure~\ref{fig:cycle-graph} by adding one or more incoming edges or adding one or more outgoing edges. 
%the $n$-compartment Fin graph ${\rm Wing}_n$
%or 
%the $n$-compartment Wing graph ${\rm Wing}_n$.
Then
$(G, In, Out, Leak)$ 
with $In=Out=\{1\}$ and $Leak=\emptyset$, 
%{\color{blue} (proper) submodels of }the $n$-compartment Fin and Wing models (Figure~\ref{fig:fin-and-wingmodel}) are 
is
generically locally identifiable.
\end{conjecture}

Earlier, we proved the case of Conjecture~\ref{conj:fin-wing} when all incoming or all outgoing edges are added  (Theorem~\ref{thm:fin-wing-0-1-leak-id}).
Next, we obtain the following partial result toward Conjecture~\ref{conj:fin-wing} (we can add one or two incoming or outgoing edges).

\begin{theorem}[Adding edges to cycle model] \label{thm:fin-wing-partial}
Let $n \geq 3$.
Let $G$ be a graph obtained from the $n$-compartment cycle graph
in Figure~\ref{fig:cycle-graph} by adding one incoming edge or adding one or (if $n \geq 4$) two outgoing edges.  Let $In=Out=\{1\}$ and $Leak \subseteq \{1,2,\dots, n \}$ with $\lvert Leak \rvert =1$.
Then the model $\widetilde{M} = (G,In,Out,Leak)$ is generically locally identifiable, and so is the model $M$ obtained by removing the leak.
\end{theorem}

\begin{proof}
By Proposition~\ref{prop:add-leak-in-out}, it suffices to prove that the models $M$ are identifiable.

First, we consider the case of a model $M$ obtained by adding one incoming edge. Let $k_{1 \ell}$ be the added edge (so, $2 \leq \ell \leq n-1$). 
%or $\{k_{1j},k_{1\ell}\}$, with ${\color{red} 2< } j<\ell\leq n-1$, be the added edge(s).  
The coefficient map $c:\mathbb{R}^{n+1}\rightarrow \mathbb{R}^{2n-2}$ 
%(or, respectively, $c:\mathbb{R}^{n+2}\rightarrow \mathbb{R}^{2n-2}$)
is obtained from the coefficient map in Proposition~\ref{prop:coeff-fin} 
by 
setting $k_{1i}=0$ and 
% (for $i\in \{2,3,\ldots,n-1\}\setminus \{\ell\}$, or, respectively, for $i\in \{2,3,\ldots,n-1\}\setminus \{j,\ell\}$),
$P_i=0$, for $i\in \{2,3,\ldots,n-1\}\setminus \{\ell\}$.
%
%\begin{packed_enum}
%    \item setting $k_{1i}=0$ and 
%% (for $i\in \{2,3,\ldots,n-1\}\setminus \{\ell\}$, or, respectively, for $i\in \{2,3,\ldots,n-1\}\setminus \{j,\ell\}$),
%$P_i=0$, for $i\in \{2,3,\ldots,n-1\}\setminus \{\ell\}$, and 
%%or, respectively, for $i\in \{2,3,\ldots,n-1\}\setminus \{j,\ell\}$). 
%    \item replacing $E^2$ by $\widetilde{E}^2=\{k_{i+1,i}~|~i\in \{2,\ldots,n\}\setminus\{\ell\}\}\cup \{k_{1\ell}+k_{\ell+1,\ell}\}$ {\color{violet} Maybe rephrase as replacing the elementary symmetric polynomials themselves, rather than their corresponding sets of variables?}. 
%%(respectively, by
%%$\widetilde{E}^2=\{k_{i+1,i}~|~i\in \{2,\ldots,n\}\setminus\{\ell\}\}\cup \{k_{1j}+k_{j+1,j}~,~k_{1\ell}+k_{\ell+1,\ell}\}$).
%\end{packed_enum}
Therefore, $n+1$ of the coefficients are 
$\widetilde{e}_1^2, \widetilde{e}_2^2, \ldots, \widetilde{e}_{n-1}^2, ~\widetilde{e}_1^2+k_{21},$  $\widetilde{\phi}_{\ell}$, 
where 
 $\widetilde{e}^2_m$ denotes the $m^{\text{th}}$ elementary symmetric polynomial on the following set: 
%  {\color{violet} does $k_{32}$ show up twice below, as $k_{32}$ and also $k_{12}+k_{32}$?}:
\[
    \widetilde{E}^2 :=
    \begin{cases}
    % CASE 1
    \{ k_{43}, \dots,  k_{n,n-1}, k_{1n},~k_{12}+k_{32}  \}
    \quad &\text{ if } \ell = 2 \\
    % CASE 2
     \{k_{32}, k_{43}, \dots, k_{\ell,\ell-1}\}
    \cup 
\{ k_{1 \ell}+ k_{\ell+1,\ell} \}
    \cup
\{ k_{\ell+2,\ell+1},\dots, k_{n,n-1}, k_{1n} \} \quad &\text{ if } \ell \geq 3~,
    \end{cases}
%\{k_{i+1,i}~|~i\in \{2,\ldots,n\}\setminus\{\ell\}\}\cup \{k_{1\ell}+k_{\ell+1,\ell}\}$, 
\]
and $\widetilde{\phi}_{\ell} := \widetilde{e}_{\ell}^2 + k_{21} \widetilde{e}^2_{\ell-1} - P_{\ell}$.

% \[
% \widetilde{\phi}_2 ~:=~
% \begin{cases}
%     % CASE 1
%     \widetilde{e}_2^2 + k_{21} \widetilde{e}^2_1 - P_2 
%     \quad &\text{ if } \ell = 2 \\
%     % CASE 2
%     \widetilde{e}_2^2 + k_{21} \widetilde{e}^2_1 
%     \quad &\text{ if } \ell \geq 3~.
% \end{cases} 
% \]

Consider the $(n+1)\times (n+1)$ submatrix of the Jacobian matrix of $c$, 
with columns indexed by $k_{32},k_{43},\ldots,k_{1n},k_{21},k_{1 \ell}$ 
and rows indexed by 
$\widetilde{e}_1^2, \widetilde{e}_2^2, \ldots, \widetilde{e}_{n-1}^2, ~\widetilde{e}_1^2+k_{21}, ~\widetilde{\phi}_{\ell}$. 
%$e_1^2,e_2^2,\ldots,e_{n-1}^2,~e_1^2+k_{21}~,~\phi_2$.
We call this matrix $\widetilde{J}$, and we let $R_1, R_2, \dots, R_{n+1}$ denote its rows. 
To show that $\det \widetilde{J} \neq 0$,
we first perform two row operations which do not change the determinant:
(1) replace $R_n$ by $(-1)R_1+R_{n}$, and 
(2) replace $R_{n+1}$ by $(-1) R_{\ell}+(-k_{21}) R_{\ell-1}+R_{n+1}$.

The resulting matrix, which we denote by $\overline{J}$, 
is block upper-triangular, with upper-left block of size $(n-1)\times (n-1)$ and lower-right block of size $(2 \times 2)$.
% {\color{blue}
%To show that $\widetilde{J}$ generically has full rank, we begin by showing that $\widetilde{J}$ is equivalent to a block upper triangular matrix.
%We apply the following row operations: 
%$(-1)\cdot R_{\ell}+(-k_{21})\cdot R_{\ell-1}+R_{n+1}\rightarrow R_{n+1}$ and 
%$(-1)\cdot R_1+R_{n}\rightarrow R_{n}$. 
%Since row operations do not change the value of the determinant, this matrix has the same determinant as $\widetilde{J}$. 
For the upper-left block, after setting $k_{1 \ell}=0$, this matrix is the Jacobian matrix of the first through $(n-1)$-st elementary symmetric polynomials (on $n-1$ variables) with respect to those variables, so by Lemma~\ref{lem:elem-sym-poly} the determinant is nonzero (and so is nonzero before setting  $k_{1 \ell}=0$).
% }
%Its upper-left block of size $(n-1)\times (n-1)$ has nonzero determinant by Lemma~\ref{lem:elem-sym-poly}. 

The lower-right $2\times 2$ block of $\overline J$ is the following matrix:
\begin{align*}
    \left[
    \begin{array}{cc}
        1 & 0 \\
        \widetilde{e}_{\ell - 1}^2 - \frac{\partial P_{\ell}}{\partial k_{21}} &
        - \frac{\partial P_{\ell}}{\partial k_{1\ell}}
    \end{array}
    \right]~,
\end{align*}
which has nonzero determinant $- \frac{\partial P_{\ell}}{\partial k_{1\ell}}=-P_{\ell}/k_{1\ell}$. 
Thus, $\det \widetilde{J}= \det \overline{J} \neq 0$, and so, by Proposition~\ref{prop:rank}, the model $M$ is generically locally identifiable.

% {\color{teal} 

Next, we consider the case of adding outgoing edges.  Let $k_{\ell1}$ (with $3 \leq \ell \leq n$) or $\{k_{j1}, k_{\ell 1} \}$ (with $3 \leq j < \ell \leq n$) be the added edge(s).  
The coefficient map 
$c: \mathbb{R}^{n +1} \to \mathbb{R}^{2n-2}$ 
(or, respectively, $c: \mathbb{R}^{n +2} \to \mathbb{R}^{2n-2}$) 
is obtained from the coefficient map in Proposition~\ref{prop:coeff-wing} by:
\begin{packed_enum}
    \item setting $k_{i1}=0$, for $i \in \{3,4,\dots,n\} \smallsetminus \{\ell\}$ (or, respectively,  $i \in \{3,4,\dots,n\} \smallsetminus \{j, \ell\}$), 
    \item setting $Q_i=0$, for $i \in \{1,2,\dots,n\} \smallsetminus \{\ell\}$ (or, respectively,  $i \in \{1,2,\dots,n\} \smallsetminus \{j, \ell\}$), and 
    \item replacing $K$ by $\widetilde K := k_{21}+ k_{\ell 1}$ (or, respectively, $\widetilde K := k_{21}+k_{j1}+k_{\ell 1}$).
\end{packed_enum}

%one outgoing edge case
% Consider the $(n+1)\times (n+1)$ submatrix of the Jacobian matrix of $c$, with columns indexed by $k_{32},k_{43},\ldots,k_{1n},k_{21},k_{j1}$ and rows indexed by the coefficients $e'_1,e'_2,\ldots,e'_{n-1},~e'_1+K~,~e'_{n+2+j}+e'_{n+1-j}K-Q_jh_0^j$. This matrix, which we call $\widetilde J$ is block lower-triangular, with upper-left block of size $(n-1)\times (n-1)$, which has nonzero determinant by Lemma~\ref{lem:elem-sym-poly}.  The lower-right block is the following $2\times 2$ matrix:
% %
% \begin{align*} 
% \left[
% \begin{array}{cc}
%     1 & 1 \\
%     e'_{n+1-j} & e'_{n+1-j}-k_{1n}k_{j+1,j}\cdots k_{n,n-1} \\
% \end{array}
% \right]
% \end{align*}
% which (it is straightforward to check) has nonzero determinant.  
% % EX: set k_{j,j+1} >>0, and all others = 1
% Thus, $\det \widetilde{J}$ is nonzero, and so, by Proposition~\ref{prop:rank}, the model $\widetilde{M}$ is generically locally identifiable.

Consider the $(n+1)\times (n+1)$ (respectively, $(n+2)\times (n+2)$) submatrix of the Jacobian matrix of $c$, with columns indexed by $k_{32},k_{43},\ldots,k_{1n},k_{21},k_{\ell 1}$ (respectively, an extra column for $k_{j1}$) 
and rows indexed by the coefficients 
$e'_1,e'_2,\ldots,e'_{n-1},~e'_1+\widetilde{K}~,~ e'_{n+2 - \ell}+e'_{n+1-\ell}\widetilde{K}-Q_{\ell}h_0^{\ell}$
(respectively, an extra row for 
% {\color{red}
$e'_{n+2 - j}+e'_{n+1-j}\widetilde{K}-Q_{\ell}h_{\ell-j}^{\ell}-Q_j h_0^j$
% }
% {\color{violet} maybe this should be $e'_{n+2-j} + e'_{n+1-j} \widetilde{K} - Q_{\ell}h_{\ell-j}^{\ell}-Q_j h_0^j$ ?} {\color{blue} Yes! Bad copy and paste error.}
). 
This matrix, which we call $\widetilde J$, is block lower-triangular.
The upper-left block, with size $(n-1)\times (n-1)$, has nonzero determinant by Lemma~\ref{lem:elem-sym-poly}.  The lower-right block is the following $2\times 2$ matrix:
\begin{align*} 
\left[
\begin{array}{cc}
    1 & 1 \\
    e'_{n+1-\ell} & e'_{n+1-\ell} - k_{1n} (k_{\ell+1,\ell} k_{\ell+2,\ell+1} \cdots k_{n,n-1}) \\
\end{array}
\right]~,
\end{align*}
which has nonzero determinant, 
%$- k_{1n} (k_{\ell+1,\ell} k_{\ell+2,\ell+1} \cdots k_{n,n-1})$, 
or (respectively) the following $3\times 3$ matrix:
\begin{align*} 
\left[
\begin{array}{ccc}
%k_21 & k_\ell 1 & k_j1
    1 & 1 & 1\\
    e'_{n+1-\ell} 
    & e'_{n+1-\ell} - k_{1n} (k_{\ell+1,\ell} k_{\ell+2,\ell+1} \cdots k_{n,n-1})
    & e'_{n+1-\ell}\\
    e'_{n+1-j} 
    & e'_{n+1-j}-k_{1n} \left(k_{\ell+1,\ell} k_{\ell+2,\ell+1} \cdots k_{n,n-1}\right)h_{\ell-j}^{\ell} 
    & e'_{n+1-j}-k_{1n} \left(k_{j+1,j} k_{j+2,j+1} \cdots k_{n,n-1}\right)
\end{array}
\right]~,
\end{align*}
\noindent which %(it is straightforward to check) 
has (nonzero) determinant 
% {\color{red}
$(Q_j Q_{\ell})/(k_{\ell 1} k_{j1}) = k_{1n}^2  (k_{\ell+1,\ell} \cdots k_{n,n-1})\left(k_{j+1,j}\cdots k_{n,n-1}\right)$.
% }
% {\color{violet}
% Maybe instead 
% $(Q_j Q_{\ell})/(k_{\ell 1} k_{j1}) = k_{1n}^2  (k_{\ell+1,\ell} \cdots k_{n,n-1})\left(k_{j+1,j}\cdots k_{n,n-1}\right)$ ?
% } {\color{blue} Yes! Absolutely.}
Thus, $\det \widetilde{J} \neq 0$. So, by Proposition~\ref{prop:rank}, the model $M$ is generically locally identifiable.
% }
\end{proof}

\section{Discussion} \label{sec:discussion}
Despite much progress, the following basic question remains open: Which linear compartmental models are identifiable?  Here, we proved that certain infinite families belong to this class, including all cycle models with up to one leak and at least one input (and at least one output).  We also showed that adding certain incoming or outgoing edges (for instance, all incoming edges or all outgoing edges) in cycle models also preserves identifiability.  

Our results give rise to several open problems.  First, for cycle models with two or more leaks, which are identifiable?  
% {\color{blue} 
We obtained results toward an answer in Section~\ref{sec:add-leaks-to-cycle}, but our knowledge is not yet complete.
% }
Next, consider cycle models with some incoming or outgoing edges added.  
We conjectured that identifiability is preserved when more than one or two incoming or outgoing edges are added (Conjecture~\ref{conj:fin-wing}).  
Next,
is identifiability preserved when the input or output is moved?  
Finally, among models containing at least one incoming edge and at least one outgoing edge, which are identifiable?

%adding more than one or two incoming or outgoing edges also preserves identifiability
%
%Next, is identifiability preserved when operations such as moving the input or output are performed on the new Fin or Wing models? Lastly, what happens when some of the ``incoming'' or ``outgoing'' edges are reversed?  We give the names Augmented-Fin and Augmented-Wing to these models, and conjecture that they are identifiable.
%\begin{conjecture}
%Assume $n\geq 3$. Every $n$-compartment Augmented-Fin and Augmented-Wing model with 
%$In=Out=\{1\}$ is generically locally identifiable. 
%\end{conjecture}
%{\color{violet} may need to edit this paragraph at the end} For these augmented models, the left-hand side of the input-output equation is the same as in Theorem \ref{Fin Coefficient Map}. Thus, the $W$ block matrix of the Jacobian submatrix, $\Tilde{J}(c)$, is the same, and since none of the incoming or outgoing edges are in $E'$, the $X$ matrix remains the zero matrix. Therefore, one only need show that $\det(Z)\neq 0$. 

In summary, as in~\cite{GHMS19One}, we view our work as a case study into the effect on identifiability of adding, removing, or moving parts of the model (input, output, edge, or leak).  
Indeed, we showed for many models that these operations preserve identifiability.  
Therefore, our work and further progress in this direction will help to resolve the fundamental problem of classifying and characterizing identifiable models.

\subsection*{Acknowledgements}
This research was initiated by SG in the REU in the Department of Mathematics at Texas A\&M University, which was funded by the NSF (DMS-1757872).  NO and AS were partially supported by the NSF (DMS-1752672).  We thank Gleb Pogudin  for helpful discussions and suggestions which improved this work.  

We are grateful to Nicolette Meshkat, who conveyed to us the proof of 
Theorem~\ref{thm:fin-wing-0-1-leak-id}, which we had stated as a conjecture in an earlier preprint. 
% {\color{blue}
We thank a referee for useful suggestions, especially for alternate proofs of Lemma~\ref{lem:elem-sym-poly}, Proposition~\ref{Output-Cycle IDable}, and Theorem~\ref{thm:fin-wing-0-1-leak-id}.
% }

%\bibliographystyle{unsrt}
\bibliographystyle{plain}
\bibliography{Apr2020-Identifiability}

\begin{thebibliography}{10}

\bibitem{baaijens-draisma}
Jasmijn~A Baaijens and Jan Draisma.
\newblock On the existence of identifiable reparametrizations for linear
  compartment models.
\newblock {\em SIAM J.\ Appl.\ Math.}, 76(4):1577--1605, 2016.

\bibitem{bearup}
Daniel~J. Bearup, Neil~D. Evans, and Michael~J. Chappell.
\newblock The input–output relationship approach to structural
  identifiability analysis.
\newblock {\em Comput.\ Meth.\ Prog.\ Bio.}, 109(2):171--181, 2013.

\bibitem{daisy}
Giuseppina Bellu, Maria~Pia Saccomani, Stefania Audoly, and Leontina D'Angi\`o.
\newblock {DAISY}: A new software tool to test global identifiability of
  biological and physiological systems.
\newblock {\em Comput.\ Meth.\ Prog.\ Bio.}, 88(1):52--61, 2007.

\bibitem{chapman-et-al}
Airlie Chapman, Marzieh Nabi-Abdolyousefi, and Mehran Mesbahi.
\newblock Controllability and observability of network-of-networks via
  cartesian products.
\newblock {\em IEEE T.\ Automat.\ Contr.}, 59(10):2668--2679, 2014.

\bibitem{chis2011}
Oana-Teodora Chis, Julio~R. Banga, and Eva Balsa-Canto.
\newblock Structural identifiability of systems biology models: A critical
  comparison of methods.
\newblock {\em PLoS ONE}, 6(11):e27755, 11 2011.

\bibitem{cobelli-constrained-1979}
C.~Cobelli, A.~Lepschy, and G.~Romanin~Jacur.
\newblock Identifiability results on some constrained compartmental systems.
\newblock {\em Math.\ Biosci.}, 47(3):173--195, 1979.

\bibitem{distefano-book}
J.~J. DiStefano, III.
\newblock {\em Dynamic systems biology modeling and simulation}.
\newblock Academic Press, 2015.

\bibitem{EhrenborgRota}
Richard Ehrenborg and Gian-Carlo Rota.
\newblock Apolarity and canonical forms for homogeneous polynomials.
\newblock {\em Eur.~J.~Combin.}, 14(3):157--181, 1993.

\bibitem{glad}
S.~T. Glad.
\newblock {\em Differential Algebraic Modelling of Nonlinear Systems}, pages
  97--105.
\newblock Birkh{\"a}user Boston, Boston, MA, 1990.

\bibitem{godfrey}
Keith Godfrey.
\newblock {\em Compartmental Models and their Application}.
\newblock Academic Press, 1983.

\bibitem{GHRS}
Elizabeth Gross, Heather Harrington, Zvi Rosen, and Bernd Sturmfels.
\newblock Algebraic systems biology: a case study for the {W}nt pathway.
\newblock {\em Bull.\ Math.\ Biol.}, 78(1):21--51, 2016.

\bibitem{GHMS19One}
Elizabeth Gross, Heather~A. Harrington, Nicolette Meshkat, and Anne Shiu.
\newblock Linear compartmental models: input-output equations and operations
  that preserve identifiability.
\newblock {\em SIAM J.\ Appl.\ Math.}, 79(4):1423--1447, 2019.

\bibitem{GNA17Three}
Elizabeth Gross, Nicolette Meshkat, and Anne Shiu.
\newblock Identifiability of linear compartment models: the singular locus.
\newblock {\em Preprint, {\tt arXiv:1709.10013}}, 2017.

\bibitem{Macdonald}
I.~G. Macdonald.
\newblock {\em Symmetric functions and Hall polynomials}.
\newblock Oxford University Press, 1995.

\bibitem{MeshkatSullivant}
Nicolette Meshkat and Seth Sullivant.
\newblock Identifiable reparametrizations of linear compartment models.
\newblock {\em J.\ Symbolic Comput.}, 63:46--67, 2014.

\bibitem{MSE15Two}
Nicolette Meshkat, Seth Sullivant, and Marisa Eisenberg.
\newblock Identifiability results for several classes of linear compartment
  models.
\newblock {\em B.\ Math.\ Biol.}, 77(8):1620--1651, 2015.

\bibitem{miao}
Hongyu Miao, Xiaohua Xia, Alan~S Perelson, and Hulin Wu.
\newblock On identifiability of nonlinear {ODE} models and applications in
  viral dynamics.
\newblock {\em SIAM Rev.}, 53(1):3--39, 2011.

\bibitem{Ovchinnikov-Pogudin-Thompson}
Alexey Ovchinnikov, Gleb Pogudin, and Peter Thompson.
\newblock Input-output equations and identifiabilty of linear {ODE} models.
\newblock {\em Preprint, {\tt arXiv:1910.03960}}, 2019.

\bibitem{vajda1982}
S.~Vajda.
\newblock Structural equivalence and exhaustive compartmental modeling.
\newblock {\em Math. Biosci.}, 69:57--75, 1984.

\bibitem{van-den-Hof}
JM~Van~den Hof.
\newblock Structural identifiability of linear compartmental systems.
\newblock {\em IEEE T.\ Automat.\ Contr.}, 43(6):800--818, 1998.

\bibitem{vicini-mam-cat}
Paolo Vicini, Hsiao-Te Su, and Joseph~J Distefano~III.
\newblock Identifiability and interval identifiability of mammillary and
  catenary compartmental models with some known rate constants.
\newblock {\em Math.\ Biosci.}, 167(2):145--161, 2000.

\end{thebibliography}
\end{document}